\documentclass[a4paper,11pt]{scrartcl}

\usepackage{graphicx}
\usepackage[utf8]{inputenc}
\usepackage[english]{babel}
\usepackage{amsmath,amssymb,enumerate,url,tabularx,stmaryrd,mathrsfs}

\usepackage{todonotes,comment,afterpage}
\usepackage{subcaption}
\usepackage{framed}
\usepackage{mathtools}
\usepackage{fullpage}
\usepackage[title]{appendix}
\usepackage{nicefrac}
\usepackage{numprint}

\usepackage{amsthm}
\usepackage{graphicx}
\usepackage{hyperref}

\theoremstyle{plain}
\newtheorem{theorem}{Theorem}[section]
\newtheorem{lemma}[theorem]{Lemma}
\newtheorem{corollary}[theorem]{Corollary}
\newtheorem{proposition}[theorem]{Proposition}

\newtheorem{assumption}[theorem]{Assumption}

\newtheorem{definition}[theorem]{Definition}

\newtheorem{remark}[theorem]{Remark}

\numberwithin{equation}{section}
\numberwithin{figure}{section}
\numberwithin{table}{section}

\newcommand{\R}{\mathbb{R}}
\newcommand{\C}{\mathbb{C}}
\newcommand{\N}{\mathbb{N}}

\newcommand{\bigO}{\mathcal{O}}

\newcommand{\abs}[1]{\left|#1\right|}

\newcommand{\support}{\operatorname{supp}}
\newcommand{\dom}{\operatorname{dom}}
\newcommand{\laplace}{\Delta}

\newcommand{\tracejump}[1]{{\left\llbracket \gamma #1 \right \rrbracket_{\Gamma}}}
\newcommand{\tracejumpsmall}[2][]{{#1\llbracket \gamma #2 #1 \rrbracket_{\Gamma}}}
\newcommand{\normaltjump}[1]{{\left\llbracket \gamma_{\nu} #1 \right \rrbracket_{\Gamma}}}
\newcommand{\normaljump}[1]{{\left\llbracket \partial_{\nu} #1 \right \rrbracket}_{\Gamma}}

\newcommand{\timejump}[2]{{\llbracket  #1  \rrbracket}_{#2}}

\newcommand{\tracemean}[1]{\left\{\!\!\left\{ \gamma #1 \right\}\!\! \right\}_{\Gamma}}
\newcommand{\normalmean}[1]{\left\{\!\!\left\{ \partial_\nu #1 \right\}\!\! \right\}_{\Gamma}}
\newcommand{\normaltmean}[1]{\left\{\!\!\left\{ \gamma_\nu #1 \right\}\!\! \right\}_{\Gamma}}

\newcommand{\BCop}{\mathscr{B}}
\newcommand{\Hdiv}[1]{H(\operatorname{div},#1)}

\newcommand{\lifting}{\mathcal{L}}

\def\VV{\mathcal{V}}

\def\TT{\mathcal{T}}
\def\SS{\mathcal{S}}
\def\PP{\mathcal{P}}

\def\II{\mathcal{I}_p}
\def\ii{i}

\newcommand{\bs}[1]{\boldsymbol #1}

\def\utot{u^{\mathrm{tot}}}
\def\uinc{u^{\mathrm{inc}}}
\def\duinc{\dot{u}^{\mathrm{inc}}}

\def\AA{A}
\def\BB{B}

\def\tstar{t_\star}
\def\pstar{p_\star}

\def\HH{\mathbb{H}}
\def\VV{\mathbb{V}}
\def\MM{\mathbb{M}}

\renewcommand{\Re}{\operatorname{Re}}
\renewcommand{\Im}{\operatorname{Im}}

\usepackage{tikz}
\usetikzlibrary{shapes}
\usepackage{pgfplots}
\usetikzlibrary{calc}

\usetikzlibrary{positioning}
\usetikzlibrary{arrows}
\usetikzlibrary{calc}
\usetikzlibrary{intersections, pgfplots.fillbetween,patterns}
\usepgfplotslibrary{colorbrewer}

\iftrue 
\newcommand{\includeTikzOrEps}[1]{ \include{figures/#1}}  

\else
\newcommand{\includeTikzOrEps}[1]{\includegraphics{figures_pdf/#1}}
\fi

\title{A $p$-version of convolution quadrature in wave propagation }

\author{Alexander Rieder\footnote{
Institute for Analysis and Scientific Computing, TU Wien, Vienna, Austria,
\texttt{alexander.rieder@tuwien.ac.at}}}
\date{\today}

\begin{document}
\maketitle
\begin{abstract}
  We consider a novel way of discretizing wave scattering problems using the general
  formalism of convolution quadrature, but instead of reducing the timestep size ($h$-method),
  we achieve accuracy by increasing the order of the method ($p$-method).
  We base this method on discontinuous Galerkin timestepping and use the Z-transform.
  We show
  that for a certain class of incident waves, the resulting schemes observes
  (root)-exponential convergence rate with respect to the number of
  boundary integral operators that need to be applied. Numerical experiments
  confirm the finding.
\end{abstract}

\section{Introduction}
Many natural phenomena are modeled by propagating waves. An important subclass
of such problems is the scattering of a wave by a bounded obstacle.
In such problems, the domain of interest is naturally unbounded and this poses
a significant challenge for the discretization scheme. One way of
dealing with such unbounded domains is via the use of boundary integral equations.
These methods replace the differential equation on the unbounded domain with
an integral equation which is posed only on the boundary of the scatterer.

For stationary problems, boundary integral techniques can be considered
well-developed (\cite{sauter_schwab,mclean}). 
If one wants to study the time dynamics of the scattering process, the integral
equation techniques become more involved. The two most common ways of discretization are the space-time Galerkin BEM~(see e.g.~\cite{BamH,BamH2,JR17} and references therein) and the convolution quadrature~(CQ) introduced
by Lubich in~\cite{L88b,L88a}. In order to allow for higher order schemes,
CQ was later extended to Runge-Kutta based methods in~\cite{LO93}.
All of these methods have in common that accuracy is achieved by reducing the timestep
size, leading to an algebraic rate of convergence.
In the context of finite element methods, a different paradigm has
proven successful, the $p$- and $hp$ version of FEM achieves accuracy by 
increasing the order of the method. Under certain conditions, it can be shown
that this procedure leads to exponentially convergent schemes.
Similarly, it has been shown that for ODEs and some PDEs, a time
discretization based on this idea and the formalism of
discontinuous Galerkin methods can also achieve exponential
convergence for time dependent problems \cite{schoetzau_diss,SS00,SS00b}.

The goal of this
work is to transfer the ideas from the $p$-FEM to the convolution quadrature.
In order to do so, we first repeat the construction of Runge-Kutta convolution quadrature
using the language of discontinuous Galerkin timestepping. We then analyze the
resulting scheme using the equivalence principle that CQ-based BEM is the same as
timestepping discretization of the underlying semigroup, an approach that
has been successfully used in the context of analyzing standard CQ~\cite{RSM20,MR17,nonlinear_wave,comp_scatter,banjai_laliena_sayas_kirchhoff_formulas}.
For spacetime Galerkin BEM, the work \cite{GCSS19} considers a similar $p$-type refinement for discretizing the wave equation, but focuses on obtaining algebraic convergence rates
for non-smooth problems.

The paper is structured as follows. In Section~\ref{sect:model_problem} we introduce the scattering
problem together with the notation used throughout the paper.
Section~\ref{sect:bem_stuff} then recalls the most important results on boundary element methods and
introduces the continuous boundary integral equations which will be discretized later on.
Section~\ref{sect:DG} is devoted to introducing the discontinuous Galerkin method
for ordinary differential equations before Section~\ref{sect:CQ} can finally
introduce the novel $p$-convolution quadrature method in a general way.
Before we can analyze the discretizations of a scattering problem, Section~\ref{sect:abstract_waves}
recalls the abstract formalism of wave propagation problems from~\cite{HQSS17} and
provides a stability and error analysis for the DG discretization thereof which is of
independent interest.
Section~\ref{sect:analysis} then finally applies all the previous results to
discretize and rigorously analyze the $p$-CQ discretization of the model scattering problem, culminating in
root-exponential error decay in Theorem~\ref{thm:conv_traces_smooth_dirichlet}
and Corollary~\ref{cor:conv_traces_smooth}. 
Sections~\ref{sect:practical_aspects} and~\ref{sect:numerics} round out the paper with some
insight into the practical implementation of the method as well as expansive numerical tests confirming the theoretical findings.

\section{Model problem}
\label{sect:model_problem}
We consider the following scattering problem.
Let $\Omega \subseteq \R^d$ be a bounded Lipschitz domain
with boundary $\Gamma:=\partial \Omega$. We write
$\Omega^+:=\R^d \setminus \overline{\Omega}$ for the exterior.
Then, the total field is given as the solution
\begin{align*}
  \partial_t^{2} \utot - \laplace \utot
  &= 0
    \qquad  \text{in $\R \times \Omega^+$}, \\
  \BCop u(t)&=0 \qquad \;\,\forall t \in \R,
\end{align*}
where the linear operator $\BCop$ encodes different boundary conditions.
Given an incident field $\uinc$, we assume that for negative times
the wave has not reached the scatterer, i.e.,
$\utot(t)=\uinc(t)$ for $t \leq 0$. We make the following
assumptions on the incident field:  it solves the homogeneous
wave equation and for negative times it does not reach the scatterer, i.e.,
\begin{align*}
  \partial_t^{2} \uinc - \laplace \uinc &=0 \qquad\;\;\qquad \text{in } \R \times \R^d,\\
  \support(\uinc(t))&\subseteq  \R^d \setminus \overline{\Omega} \qquad \text{ for $t\leq 0$}.
\end{align*}
Using these assumptions, we write $\utot=\uinc + u$, where the scattered
field $u: [0,T] \to H^1(\Omega^+)$ is given as the solution to
\begin{align}
  \label{eq:wave_eqn_classic}
  \partial_t^2 u - \laplace u
  &= 0,
    \qquad
    \BCop u=-\BCop \uinc, \qquad
    \qquad u(0)=\dot{u}(0)=0.
\end{align}

For an open set $\mathcal{O}$, we use the usual Sobolev spaces $H^s(\mathcal{O})$
for $s \in \N_0$. On the boundary $\Gamma$, we use the fractional order
space $H^{1/2}(\Gamma)$ and its dual $H^{-1/2}(\Gamma)$. See for example
\cite{mclean} for a detailed treatment. The space $H(\operatorname{div},\mathcal{O})$
denotes the space of $L^2(\mathcal{O})$-functions with (distributional) divergence in $L^2(\mathcal{O})$.
We write $\langle \cdot,\cdot\rangle_{\Gamma}$ for the extended $L^2$-product to
$H^{-1/2}(\Gamma) \times H^{1/2}(\Gamma)$.

In order to be able to apply boundary integral techniques, we need some trace and jump operators. We write $\gamma^-$ for the interior Dirichlet trace operator and $\gamma^+$ for
the exterior. Similarly, we write $\gamma_{\nu}^{\pm}$ for the normal trace
and $\partial_\nu^\pm$ for the normal derivative, where in
all cases we take the normal vector to point out of $\Omega$. 
The jumps and mean values are defined as
\begin{alignat*}{6}
  \tracejump{u}&:=\gamma^- u - \gamma^+u,
  \qquad &\normaljump{u}&:=\partial_\nu^- u -\partial_\nu^+ u,
  \qquad &\normaltjump{\bs w}&:=\gamma_\nu^- \bs u -\gamma_\nu^+ \bs u,
  \\
  \tracemean{u}&:=\frac{1}{2}(\gamma^+u + \gamma^- u), \qquad &\normalmean{u}&:=\frac{1}{2}(\partial_\nu^+u + \partial_\nu^- u)
  \qquad &\normaltmean{ \bs w}&:=\frac{1}{2}(\gamma_\nu^- \bs w +\gamma_\nu^+ \bs w).
\end{alignat*}

\subsection{A brief summary of time domain boundary integral equations}
\label{sect:bem_stuff}
In this section we give a brief introduction to boundary integral equation methods.
For more details we refer to the monographs~\cite{book_sayas,book_banjai}.

We start in the frequency domain. For $s \in \C_{+}:=\{s \in \C: \; \Re(s)>  0\}$, the fundamental solution for the differential operator $\laplace - s^2$
is given by
  \begin{align*}
    \Phi(z;s):=\begin{cases}
      \frac{\ii}{4} H_0^{(1)}\left(\ii s \abs{z}\right), & \text{ for $d=2$}, \\
      \frac{e^{-s\abs{z}}}{4 \pi \abs{z}}, & \text{ for $d = 3$},
    \end{cases}
  \end{align*}
  where $H_0^{(1)}$ denotes the Hankel function of the first kind and order zero. Using this function,
  we can define the single- and double layer potential operators
\begin{align*}
    \left( S(s) \lambda \right)(\mathbf x)
    	&:=\int_{\partial \Omega}{\Phi(\mathbf x-\mathbf y;s) \lambda(\mathbf y) 
    	\;d\sigma(\mathbf y)},  \quad
    \left( D(s) \phi \right)(\mathbf x)
   :=\int_{\partial {\Omega}}{\partial_{\nu(\mathbf y)}
    	\Phi(\mathbf x-\mathbf y;s) \phi(\mathbf y) \;d\sigma(\mathbf y)},
\end{align*}
where $d\sigma$ is the arc/area element on $\partial\Omega$.
The time domain version of these operators, often called the retarded potentials,
can be defined using the inverse Laplace transform
\begin{alignat*}{6}
   S(\partial_t)\lambda &
  :=\mathscr{L}^{-1}\big(  S(\cdot) \mathscr{L} \{\lambda\} \big),\quad\text{and} \quad
   D(\partial_t)\phi
  :=\mathscr{L}^{-1}\big(  D(\cdot) \mathscr{L} \{\phi\} \big).
\end{alignat*}

Using these potentials, we define:
\begin{align*}
  V:=\gamma^{\pm} S,& \qquad  K:=\tracemean{ D}, 
                      \qquad K^t:= \normalmean{ S },  \qquad  W:=-\partial_{\nu}^{\pm} D,
\end{align*}
where our omission of parameters $s$ or $\partial_t$ is to be understood in the sense that
the statement holds for both the frequency- and time-domain version of these operators.

We further generalize the Laplace domain boundary integral operators to the case of matrix-valued
wave numbers using the following definition:
\begin{definition}
  \label{def:functional_calculus}
  Let $\mathcal{X},\mathcal{Y}$ be Banach spaces
   and
    $L(\mathcal{X},\mathcal{Y})$ denotes all bounded linear operators
    between these spaces. For a matrix $\bs A \in \R^{m\times m}$ and a holomorphic function
  $f: \mathcal{O} \to L(\mathcal{X},\mathcal{Y})$ where $\mathcal{O}$ is an open neighborhood of $\sigma(\bs A)$,
  we define the operator $f(A): \mathcal{X}^m \to \mathcal{Y}^m$ as:
  \begin{align}
    \label{eq:def:functional_calculus}
    f(\bs A):=\frac{1}{2\pi \ii}\int_{\mathcal{C}}{\big(\bs A - \lambda\big)^{-1} \otimes f(\lambda)\,d\lambda}
  \end{align}
  with a contour $\mathcal{C}$ encircling the spectrum $\sigma(\bs A)$ with winding number
  $1$. The operator $\otimes$ denotes the Kronecker product, i.e.,
  \begin{align*}
    A \otimes f := \begin{pmatrix} a_{11} f & \dots & a_{1m} f\\
      \vdots & \dots & \vdots \\
      a_{m1} f& \dots & a_{mm} f.
    \end{pmatrix}
  \end{align*}
\end{definition}
We refer to \cite[Chapter VIII.7]{Yos80}] or \cite[Chapter 11]{GV13} for more details on this calculus.

It is well known that one can use these boundary element methods to solve
time domain scattering problems. We consider the following  possibilities
\begin{proposition}
  \label{prop:different_model_problems}
  Assume that we are in one of the following cases:
  \begin{enumerate}
  \item \textbf{Dirichlet problem:}
    If $\BCop(u)=\gamma^+ u$,
    we can discretize using a \textbf{direct method},
    i.e., let $\lambda$ solve
    $$
    \partial_t V(\partial_t) \lambda = - \Big( -\frac{1}{2} + K(\partial_t)\Big) \gamma^+\duinc
    \quad \text{and set}\quad
    u = -S(\partial_t) \lambda - \partial_t^{-1} D(\partial_t) \gamma^+\duinc.
    $$    
    or, we can also  discretize using an \textbf{indirect method},
    i.e., let $\lambda$  solve
    $$
    \partial_t V(\partial_t) \lambda = -\gamma^+ \duinc
    \quad \text{and set}\quad
    u =  S(\partial_t) \lambda.
    $$
  \item \textbf{Neumann problem:}
    For $\BCop(u)=\partial_{\nu}^+ u$,  using a \textbf{direct method},
    means solving for $\psi$ such that
    $$
    W(\partial_t) \psi = - \Big(\frac{1}{2} + K^t(\partial_t)\Big) \partial_{\nu} \uinc
    \quad \text{and set}\quad
    u = S(\partial_t) \partial_{\nu} \uinc + D(\partial_t) \psi.
    $$
    Alternatively,  the \textbf{indirect method}
    solves for $\psi$ such that
    $$
    W(\partial_t) \psi =  -\partial_{\nu} \uinc
    \quad \text{and sets}\quad
    u = -D(\partial_t) \psi.$$
  \end{enumerate}
  Then the function $u$ defined by the  potentials satisfies
  the model problem~\eqref{eq:wave_eqn_classic}. 
\end{proposition}
\begin{proof}
  See~\cite{book_sayas} for a rigorous treatment of the retarded potentials.
\end{proof}
\begin{remark}
  Compared to the more standard formulation for the Dirichlet problem
  involving $\uinc$,
  we used $\duinc$ as the data instead. This leads to a
  more natural formulation using the first order wave equation setting
  and its discretization. This is similar to what was done in~\cite{comp_scatter}.
  See also Remark~\ref{rem:why_differentiate}.
\end{remark}

\section{Discontinuous Galerkin timestepping}
\label{sect:DG}
Following ideas and notation of \cite{schoetzau_diss,SS00}, our time
discretization scheme will be based on discontinuous Galerkin timestepping
which we now introduce.
We fix $h>0$ and consider the infinite time grid $t_j:=j h$ for $j \in \N_0$,
and set $\TT_{h}:=\{ (t_j,t_{j+1}), \; j \in \N_0\}.$

We  define one-sided limits and jumps. For
a function $u$ which is continuous on each element $(t_j,t_{j+1}) \in \TT_h$
but may be globally discontinuous, we set
\begin{align*}
  u(t_n^+):=\lim_{ \small \substack{h \to 0, \\h>0}} u(t_n+h),\qquad
  u(t_n^-):=\lim_{  \small\substack{h \to 0, \\h>0}} u(t_n-h), \qquad
  \timejump{u}{t_n}:=u(t_n^+)-u(t_n^-).
\end{align*}

Denoting by $\PP_p$ the space of all polynomials of degree $p$, we then define two spaces of piecewise polynomials
\begin{align*}
  \SS^{p,0}(\TT_{h})&:=\{u \in L_{\mathrm{loc}}^2(\R_+): u|_{K} \in \PP_p \qquad \forall K \in \TT_{h}\}, \\
  \widetilde{\SS}^{p,0}(\TT_{h})&:=\{ u \in \SS^{p,0}(\TT_h): \operatorname{supp}(u) \text{ is  bounded} \}.
\end{align*}
By convention, we will take functions in $\SS^{p,0}(\TT_h)$ and similar
piecewise smooth functions to be left-continuous, i.e., we set
$u(t_n):=u(t_n^-)$. In addition, we will use the standard notations $\dot{u}$ or $u'$ also
for piecewise differentiable functions. The derivative is to be understood to be taken in each element
separately.
Furthermore, we will write $u \in \SS^{p,0}(\TT_h) \otimes \mathcal{X}$ for piecewise polynomial functions
with values in a Banach space $\mathcal{X}$, i.e., $u(t) \in \mathcal{X}$ for all $t \geq 0$.

An important part will be played by the following projection operator:
\begin{proposition}[{\cite[Def. 3.1]{SS00b}}]
  \label{prop:def_II}
  Let $u \in L^{\infty}(\R_+)$
  be left-continuous at the nodes $t_j$ of $\TT_h$.
  For fixed $p \in \N_0$, define
  $\II u \in \SS^{p,0}(\TT_h)$ by the conditions
  \begin{subequations}
  \begin{align}
    \II u(t_j^-)&=u(t_j^-) \qquad\qquad\qquad \forall j \in \N,
                  \label{eq:prop:def_II:interp} \\
    \int_{0}^{\infty}{\II u(t) \, v(t)\,dt}
    &=
      \int_{0}^{\infty}{ u(t) v(t)\, dt} \quad\qquad \forall v \in \widetilde{\SS}^{p-1,0}(\TT_h).
      \label{eq:prop:def_II:ortho}
  \end{align}
\end{subequations}
Then the following holds:
\begin{enumerate}[(i)] 
\item
  \label{it:prop:def_II:0}
  $\II u $ is a projection, i.e., $\II u =u$ for $u \in \SS^{p,0}(\TT_h)$.
  \item
    \label{it:prop:def_II:1}
    Fix $N\in \N$, $s \in \N_0$ and set $T:=N h$.
    If $p \geq s$ and $u \in W^{s+1,\infty}(0,T)$
    $$
      \| u - \II u\|_{L^{\infty}(0,T)}\lesssim \left(\frac{k}{2}\right)^{s+1}\sqrt{\frac{\Gamma(p+1-s)}{\Gamma(p+1+s)}} \|u\|_{W^{s+1,\infty}(0,T)}.
      $$
      The constant is independent of $s$, $p$, $k$, $T$, and $u$.
    \item
      \label{it:prop:def_II:2}
      The following stability holds      
      \begin{align*}
        \| \II u\|_{L^{\infty}(0,T)}
        \lesssim \|u\|_{L^{\infty}(0,T)} + k \|u'(t)\|_{L^{\infty}(0,T)}.
      \end{align*}
      The constant is independent of $p$, $k$, $u$ and $T$.
    \end{enumerate}
    All statements naturally extend to Banach space valued functions.
\end{proposition}
\begin{proof}
  The  operator is the same as in
  {\cite[Def. 3.1]{SS00b}}, except that we opted for a
  global definition instead of defining it element by element.
  The approximation property (\ref{it:prop:def_II:1}) can be found in
  \cite[Corollary 3.10]{SS00}. The independence of $T$ can be found by a
  scaling argument.
  Statement (\ref{it:prop:def_II:2})
  follows from~(\ref{it:prop:def_II:1}) by taking $s=0$ and applying
  the triangle inequality
  \begin{align*}
    \|\II u\|_{L^{\infty}(0,T)}
    &\leq      \|u\|_{L^{\infty}(0,T)} +  \|u-\II u\|_{L^{\infty}(0,T)}
    \lesssim \|u\|_{L^{\infty}(0,T)} +   k\|u'\|_{L^{\infty}(0,T)}.
    \qedhere
  \end{align*}  
\end{proof}

To introduce the DG timestepping, we consider the following simple linear ODE
for $\zeta \in \C$:
\begin{align}
  \label{eq:simple_ode}
y'(t)= \zeta y(t) + g(t) \qquad \forall t> 0\qquad \xi(0) =0.
\end{align}
Then, the DG approximation \eqref{eq:simple_ode} is given by $y^h \in \SS^{p,0}(\TT_h)$
such that
  \begin{align}
        \label{eq:ode_dg_formulation}
      &\int_{0}^{\infty}{ \big(\dot{y}^h(t) - \zeta y^{h}(t) \big)v^h(t)\,dt}
        +  \sum_{n=1}^{\infty}
        {\llbracket  y^h \rrbracket_{t_n} v^h(t_n^+)} =
        \int_{0}^{t}{g(t) v^h(t) \,dt}
         \qquad   \forall v^h  \in \widetilde{\SS}^{p,0}(\TT_h).
  \end{align}

Instead of this global space-time perspective, one can rewrite
\eqref{eq:ode_dg_formulation} as a timestepping method.
Set ${y}^n(\tau):=y^h(t_n+ h \tau) \in \PP_p$ and
$g^n:=g(t_n + h \tau)$ for all $n \in \N_0$.
Then, this sequence of
functions solves for all $q \in \PP_p$:
\begin{align}
  \label{eq:dg_as_timestepping_ode}
  &\int_{0}^{1}{
    \dot{y}^n(\tau)\,  q(\tau)
    -h \zeta y^n(\tau)\,  q(\tau)
  \,d\tau}
    + y^n(0)q(0)
  =  y^{n-1}(1)q(0) +h\int_{0}^{1}{g^n(\tau) q(\tau)\,d\tau}.
\end{align}

If we take $(\varphi_j)_{j=0}^{p}$ as a basis of $\mathcal{P}_p$,
we can introduce the matrices
\begin{align}
  \mathbf{M}_{ij}&:=\int_{0}^{1}{\varphi_{j}(\tau) \varphi_i(\tau)\,d\tau}
                   \qquad \text{and} \qquad
  \mathbf{S}_{ij}:=
     \int_{0}^{1}{\varphi_{j}'(\tau) \varphi_i(\tau)\,d\tau}
                   + \varphi_j(0) \varphi_i(0)
\end{align}
and for  $\tau \in [0,1]$ the trace operators
$  [\mathbf{T}(\tau)]_{j}:=\varphi_j(\tau) \qquad \text{in } \;\R^{1 \times p+1}$.

The method~\eqref{eq:dg_as_timestepping_ode} can be recast as follows.
For all $n \in \N$, find $ Y^n  \in \C^{p+1}$ such that
  \begin{align}
    \label{eq:dg_as_timestepping2}
    \mathbf{S} Y^n
    -h \zeta  \mathbf{M}  Y^n
    &=  [\mathbf{T}(0)]^t \mathbf{T}(1)  Y^{n-1} + h G^n.
\end{align}
where we make the convention of $ Y^{-1}:=0$. The
connection to \eqref{eq:dg_as_timestepping_ode} is given by
the relation $y_n(\tau)=\sum_{j=0}^{p}{Y^n_j \varphi_j(\tau)}$
and $G^n_j:=\int_{t_n}^{t_{n+1}}{g^n(\tau) \varphi_j(\tau)\,d\tau}$.

\begin{remark}
Compared to the reference~\cite{schoetzau_diss}, we use the
reference interval $(0,1)$ instead of $(-1,1)$, resulting in the absence of the
factor $2$ in certain formulas. This choice will simplify some of the calculations
later on.
\end{remark}

The interpolation operator from Proposition~\ref{prop:def_II} is designed in a way
that interacts well with the DG time discretization. We formalize this in the following Lemma:
\begin{lemma}
  \label{lemma:dot_vanishes_for_eta}
  Let $u \in H^1_{loc}(\R_+)$. Then the following identity holds
  for all $v \in \widetilde{\SS}^{p,0}(\TT_h)$:
  \begin{align*}
    \int_{0}^{\infty}{(\dot{u}(t) - [\II u]'(t))v(t)\,dt}
    + \sum_{n=1}^{\infty}{\timejump{u- \II u}{t_{n}}{v(t_{n}^+)}}
    &= 0.
  \end{align*}
\end{lemma}
\begin{proof}
  We work on a single element $(t_{n},t_{n+1})$. The general result follows by summing up.
  We write $\eta:=u-\II u$.
  For $v|_{(t_n,t_{n+1})}\in \PP_p$ with $v=0$ otherwise, we get:
  \begin{align*}
    \int_{0}^{\infty}{\dot{\eta}v(t)\,dt}
    + \sum_{m=1}^{\infty}{\timejump{\eta}{t_{m}}{v(t_{m}^+)}} 
    &=
      \int_{t_{n}}^{t_{n+1}}{\dot{\eta}(t)v(t)\,dt} 
      +{\timejump{\eta}{t_{n}}{v(t_{n}^+)}}\\
      &=
      -\underbrace{\int_{t_{n}}^{t_{n+1}}{\eta(t)\dot{v}(t)\,dt}}_{=0}
      +\underbrace{\eta(t_{n+1}^-)v(t_{n+1}^-)}_{=0}
      - \eta(t_{n}^+)v(t_{n}^+) 
      +\underbrace{{\timejump{\eta}{t_{n}}{v(t_{n}^+)}}}_{=\eta(t_{n}^+) v(t_{n}^+)}
      \\
      &= 0,
  \end{align*}
  where in the last step we used the properties of $\II$ from Proposition~\ref{prop:def_II}.
  Namely that $\II u$ is left-side continuous, i.e. $\eta(t_m^-)=0$,
  and $\int{\II u \dot{v}}= \int{u \dot{v}}$ since
  $\dot{v} \in \PP_{p-1}$.
\end{proof}

\section{The $p$-version of Convolution Quadrature}
\label{sect:CQ}

\subsection{Derivation of the $p$-CQ method}
\label{sect:deriv_of_CQ}
There are multiple ways of arriving to the convolution quadrature
formalism. We introduce the $p$-CQ in  a way that leads to a general ``operational calculus'',
but the analysis will be based on an equivalence principle to a timestepping approximation.

The main tool for deriving the CQ will be the $Z$-transform. For a
sequence $y=(y^n)_{n=0}^{\infty} \subseteq \C$ we define the formal power series
$$
\widehat{y}:=\mathscr{Z}(y):=\sum_{n=0}^{\infty}{y^n z^n}.
$$
For a piecewise polynomial $g \in \SS^{p,0}(\TT_h)$ we similarly define
the polynomial $\widehat{g}(s) \in \PP_{p}$
$$
\widehat{g}(s):=\mathscr{Z}(g):=\sum_{n=0}^{\infty}{g(h(\cdot+n)) z^n}.
$$
This is equivalent to working in a given basis of polynomials and taking the $Z$-transform of the coefficient vector.

We first look at the approximation~\eqref{eq:ode_dg_formulation} of \eqref{eq:simple_ode},
and derive an explicit representation of the $Z$-transform $\mathscr{Z}[y^{h}]$.
We work in timestepping form \eqref{eq:dg_as_timestepping2}.
Set $Y^{-1}:=0$,
and $G^n:=\pi_{L^2}g|_{(t_{n},t_{n+1})}$. Then
$Y^{n} \in \R^{m+1}$ is given by, 
\begin{align*}
  \mathbf{A} Y^n - \zeta h \mathbf{M} Y^n
  &= h\mathbf{M} G^n + \bs T(0)^t\bs T(1) Y^{n-1} \qquad \forall n \in \N_0.
\end{align*}
Since $Y^{-1}=0$, we get that the index shift corresponds to multiplying by $z$. This
means for the $Z$-transform:
\begin{align*}
  \mathbf{A} \widehat{Y} - h \zeta \mathbf{M} \widehat{Y}
  - h z \bs T(0)^t \bs T(1) \widehat{Y}
  &= h \mathbf{M} \widehat{G} 
\end{align*}
or equivalently
\begin{align*}
  \widehat{Y} &=\Big(\mathbf{A} - \zeta h \mathbf{M} - z \bs T(0)^t \bs T(1)\Big)^{-1} \mathbf{M} \widehat{G}
    =\Big(\mathbf{M}^{-1}(\mathbf{A} - z h \bs T(0)^t \bs T(1) -  \zeta I\Big)^{-1} \widehat{G}.
\end{align*}
Thus, if we define the matrix valued CQ symbol
\begin{align}
  \label{eq:def_delta}
  \bs \delta(z):=\bs M^{-1}\big(\bs S - z \bs T(0)^t \bs T(1)\big),
\end{align}
we get the explicit representation
\begin{align}
  \label{eq:expl_repr_z_transform}
  \widehat{Y}=\Big(\bs \delta\big(\frac{z}{h}\big) - \zeta\Big)^{-1} \widehat{G}.
\end{align}

From this calculation, we can see that if $(Y^n)_{n\in \N}$ is the
DG approximation of $y'=\zeta y +g$ then $\widehat{Y}$ solves
the frequency domain equation, $\frac{\bs\delta(z)}{h}\widehat{Y} = \zeta \widehat{Y} +\widehat{g}$.
Thus, if we have a solution map $K(\zeta)$ realizing $\widehat{g} \mapsto  \widehat{Y}$
for $\zeta=\nicefrac{\bs\delta(z)}{h}$, we
can recover the original DG approximation via the CQ:
\begin{align}
  \label{def_p_cq_simple1}
  Y^n=K(\partial_t^h) G := \sum_{j=0}^{n}{ \omega_{n-j} G^j}
  \qquad \text{with}
  \qquad \sum_{n=0}^{\infty}{\omega_n} z^n = K\Big(\frac{\bs \delta(z)}{h}\Big).
\end{align}

\begin{remark}
  Since we were working with the timestepping vectors  $Y^n$,
  the mass and stiffness matrices $\bs{M}$, $\bs S$,
  as well as the CQ symbol $\bs \delta$
  depend on the choice of basis $(\varphi_j)_{j=0}^{p}$ of the space
  $\PP_{p}$.
  To see that the definition is matrix-independent,
  one can also repeat the construction using $\PP_p$ valued operators.
  Most notably, we can identify $\bs \delta$  with the linear map
  \begin{align*}
    \delta(z): \PP_p \to \PP_p:
    \quad u &\mapsto \delta(z) u \qquad \text{ such that } \\
    \int_{0}^{1}{[\delta(z) u] v} &= \int_{0}^{1}{\dot{u} v} +u(0)v(0) - z u(1) v(0) \;\quad \forall v \in \PP_{p}.
  \end{align*}
  The functional calculus definition~\eqref{eq:def:functional_calculus} then
  coincides for the operator valued version and its matrix representation.
\end{remark}

We can now generalize the definition from~\eqref{def_p_cq_simple1} to general
convolutional symbols by repeating the computation in~\cite{L88a}.

For $\sigma \geq 0$, assume that
we are given $K: \C_\sigma:=\{z \in \C: \Re(z) > \sigma\} \to \C$ which is analytic
and satisfies the bound $|K(s)|\leq C |s|^{-\mu}$ for some $\mu>1$.
(The generalization to operator
valued maps is straight-forward).
For the sake of this computation, we assume that $g \in \SS^{p,0}(\TT_h)$.
Investigating the convolution with the inverse Laplace transform of $K$, we observe.
\begin{align*}
  u(t)
  &:=K(\partial_t) g
    =
    \int_{0}^{t}{\Big(    
    \frac{1}{2\pi \ii}
    \int_{\sigma+\ii \R} {K(\zeta ) e^{\zeta t}\,ds}\Big) g(t -\tau) \,d\tau} 
  =
    \frac{1}{2\pi \ii}
  \int_{\sigma+ \ii \R} {K(\zeta)
  \int_{0}^{t}{
    e^{\zeta t} g(t -\tau) \,d\tau}\,d\zeta} \\
  &=    \frac{1}{2\pi \ii}
  \int_{\sigma+ \ii \R} {K(\zeta)
    \int_{0}^{t}{
    e^{\zeta(t-\tau)} g(\tau) \,d\tau}\,d\zeta} .
\end{align*}
The inner integral is the solution to $y'=\zeta y + g$ with $y(0)=0$.
We replace this function by the DG approximation $y^{h}(t;\zeta)$
to get an approximation $u^h \in \SS^{p,0}(\TT_h)$ by
\begin{align}
  \label{eq:cq_as_integral}
  u(t)
  &\approx
    u^h(t_n):=
\frac{1}{2\pi \ii}
    \int_{\sigma+ \ii \R} {K(s) y^{h}(t;\zeta)\,d\zeta}.    
\end{align}
From~\eqref{eq:expl_repr_z_transform}, we get
that the $Z-$ transforms are given by
$$
\widehat{u}^h=
\frac{1}{2\pi \ii}
\int_{\sigma+ \ii \R} {K(s) \Big(\frac{ \delta(z)}{h} -\zeta \Big)^{-1} \widehat{g}\,ds}
= K\Big(\frac{\delta(z)}{h}\Big) \widehat{g},
$$
where in the last step we used the functional calculus definition~\eqref{eq:def:functional_calculus}.
In order to undo the Z-transform, we use the Cauchy product formula
to get:
\begin{align*}
  u^h(t_n+ h \tau) = \sum_{j=0}^{n}{ [W_j g(t_{n-j}+\cdot)](\tau) }
  \qquad \text{and} \qquad
  \sum_{n=0}^{\infty}{ W_n z^n } = K\Big(\frac{\delta(z)}{h}\Big) 
\end{align*}
where the weights $W_n$ are linear maps  $\PP_{p} \to \PP_{p}$ (or equivalently matrices in a chosen basis).
This formula is completely analogous to the standard Runge-Kutta
convolution quadrature formula from \cite{LO93}. Just as in that case, we observe that
the same formula can also be applied if $|K(s)|\leq C |s|^{\mu}$ for $\mu>0$. 
We finalize our introduction of the p-CQ method by formalizing the operational calculus notation.
\begin{definition}[Operational calculus]
  \label{def:hpcq}
    Let $\mathcal{X}$, $\mathcal{Y}$ be Banach spaces.
    Assume that there exists $\sigma >0$ such that
    $K: \C_{\sigma} \to L(\mathcal{X},\mathcal{Y})$ is holomorphic
    on $\C_{\sigma}:=\{z \in \C: \Re(z) > \sigma\}$, and that there
    exists $\mu \in \R$ such that      
    $\|K(s)\|_{L(\mathcal{X},\mathcal{Y})}
    \leq C |s|^{\mu}$ for all $s \in \C_{\sigma}$.

    Given $g \in H^1_{\mathrm{loc}}(\R_+)$ or $g \in \SS^{p,0}(\TT_h)$,
    we define $K(\partial_t^h) g \in \SS^{p,0}(\TT_h)$
    by
      \begin{align}
        \label{eq:def_of_cq_method}
        \big(K(\partial_t^h) g\big)(t_n+\cdot):= \sum_{j=0}^{n}{ W_j \, [\II g](t_{n-j}+\cdot) }
        \qquad \text{and} \qquad
        \sum_{n=0}^{\infty}{ W_n z^n } = K\Big(\frac{\delta(z)}{h}\Big).
      \end{align}      
    \end{definition}

    \begin{remark}
      In the definition of the operator calculus we used the projection operator
      $\II$  on the right-hand side in order to get a piecewise polynomial.
      One could also think of using for example the $L^2$-projection, but
      using $\II$ will turn out to be crucial for the analysis, while
      being just as simple to compute; see Section~\ref{sect:practical_aspects}.
      In practice, using the $L^2$-projection leads to slightly worse results,
      e.g., the loss of one power of $h$ in a fixed-order method.
    \end{remark}

    \begin{remark}
      One could use the representation of $CQ$ in~\eqref{eq:cq_as_integral} also
      for error-analysis. This is the avenue taken in the works~\cite{BL11,BLM11,LO93,L88a}.
      Such an approach would also be possible in this case, but in order to obtain
      reasonably sharp estimates
      akin to Theorem~\ref{thm:best_approx},
      the analysis becomes quite involved.
      Therefore we decided to use a more streamlined approach using an equivalent
      time-domain discretization of our model-problem.
      The more general black-box theory for the operational calculus
       is subject of the upcoming work~\cite{blackbox}.
    \end{remark}
    
    \subsection{Basic properties of $p-CQ$}
    We start with investigating basic properties of
    the new method. Most of the method hinges
    on understanding the CQ symbol $\bs \delta(z)$.
    We first show that we have detailed control of the 
    spectrum of this matrix/linear map.
\begin{lemma}
  \label{lemma:spectrum_of_delta}
  For $\abs{z}\leq 1$,
   the spectrum of $\bs \delta(z)$ satisfies:
  \begin{align*}
    \Re\big(\sigma(\bs \delta(z))\big) \geq \frac{1-|z|^2}{2}.
  \end{align*}
  In particular  $\bs \delta(z)$ is invertible and
  $\sigma(\bs \delta(z)) \subseteq \C_+:=\{z \in \C: \; \Re(z)>0 \}$.
  In addition, there exists a generic constant $C>0$ such that
  $$
  |\lambda|\leq C (p+1)^2 \qquad \forall \lambda \in \sigma(\bs \delta(z)).
  $$
\end{lemma}
\begin{proof}
  For $\lambda \in \sigma(\bs \delta(z))$, there exists
  $u \in \PP_{p}$ with  $\|u\|_{L^2(0,1)}=1$ such that:
  \begin{align*}
    \int_{0}^{1}{u'v} + u(0)v(0) - z u(1)v(0)&=\lambda \int_{0}^{1}{u v} \qquad \forall v \in \PP_p.
  \end{align*}
  Taking $v:=\overline{u}$, and noting that $\Re(u'\overline{u})=\frac{1}{2}\partial_t |u|^2$ gives:
  \begin{align}
    \Re(\lambda) \|u\|_{L^2(-1,1)}^2
    &=\frac{1}{2} |u(1)|^2 - \frac{1}{2}|u(0)|^2 + |u(0)|^2 - \Re(z u(1) \overline{u(0)})  \nonumber\\
    &\geq  \frac{1-|z|^2}{2} |u(1)|^2,
      \label{eq:spectrum_of_delta:est1}
  \end{align}
  where in the last step we used Young's inequality. Since $|z|\leq 1$ we see that 
  $\Re(\lambda)\geq 0$. What is needed to show the sharper estimate on  $\Re(\lambda)$ is a
  good bound on  the boundary value $u(1)$. To obtain this, we pick
  $v(t):=t \overline{u'(t)} \in \PP_p$ as the test-function to get:
  \begin{align*}
    \int_{0}^{1}{t |u'(t)|^2} &=\lambda \int_{0}^{1}{t u(t) \overline{u'(t)}\,dt }.
  \end{align*}
  Again using $\Re(u'\overline{u})=\frac{1}{2}\partial_t |u|^2$, this implies:
  \begin{align*}
    2\Re(\lambda^{-1})\int_{0}^{1}{t |u'(t)|^2}
    &=\int_{0}^{1}{t \partial_{t} |u(t)|^2\,dt}
      = t |u(t)|^2\Big|_{0}^{1} - \int_{0}^{1}{|u(t)|^2\,dt}
    = |u(1)|^2 - 1,
  \end{align*}
  where in the last step we used the normalization $\|u\|_{L^2(0,1)}=1$. Since 
  we already established $\Re(\lambda)\geq 0$, we get that the left-hand side
  is non-negative. Therefore  $|u(1)| \geq 1$. Taking this information
  and inserting it into~\eqref{eq:spectrum_of_delta:est1} gives the stated result.

  To see the upper bound on the spectrum, we use inverse estimates
  (see for example~\cite[Thm. 3.91,Thm. 3.92]{schwab_book}):
  \begin{align*}
    |\lambda| &= \Big| \int_{0}^{1} {u' \overline{u}}  + |u(0)|^2 - z u(1) \overline{u(0)} \Big| 
                \leq \|u'\|_{L^2(0,1)}\|u\|_{L^2(0,1)} + (1+|z|)\|u\|^2_{L^{\infty}(0,1)} \\
    &\lesssim (1+p)^{2} \|u\|_{L^2(0,1)}^2. \qedhere
  \end{align*}  
\end{proof}

    \begin{corollary}
      \label{corollary:hpcq_well_posed}
      Let $h$ be sufficiently small, depending on $\sigma$.
      Definition \ref{def:hpcq} is well-posed. If $\sigma=0$,
      then all $h>0$ are admissible.
    \end{corollary}
    \begin{proof}
      We apply Lemma~\ref{lemma:spectrum_of_delta}.
      For $|z|\leq z_0<1$, the matrix $\bs \delta(z)$ has spectrum in the right-half plane,
      i.e.,  $\Re(\lambda)>\frac{1-|z|^2}{2}>0$ for all eigenvalues $\lambda$.
      Thus, for $h < \frac{1-z_0^2}{2\sigma}$, the map $\delta(z)/h$ has spectrum
      in the set $\C_{\sigma}$ and $K(\delta(z)/h)$ is well-defined and analytic
      with respect to $z$ by definition~\eqref{eq:def:functional_calculus}. Thus
      the expansion as a  power series~\eqref{eq:def_of_cq_method} is well-defined.
    \end{proof}

    \begin{remark}
      \label{rem:composition_rule}
      By the same proof as for the standard CQ (see~\cite[Remark 2.18]{book_banjai})
      one can show that our discrete method describes an operational calculus in the sense
      that
      $$
      K_1(\partial_t^{h}) K_2(\partial_t^h) = (K_1 \cdot  K_2)(\partial_t^h).
      $$
      For example, if we write
      $
      \partial_t^{h} K(\partial_t^h)
      $,
      this (in practice) is computed in a single step by using $s K(s)$
      as the convolution symbol.
    \end{remark}

    We also can give a more explicit representation of the
    very basic CQ operations $\partial_t^h$ and its inverse $[\partial_t^{h}]^{-1}$,
    i.e., the operators for the symbols $K(s)=s$ and $K(s)=s^{-1}$ respectively.

    \begin{lemma}
  \label{lemma:integration}
  Let $\psi \in \SS^{p,0}(\TT_h)$.
  The discrete integral is given explicitly by
  \begin{align}
    [\partial_t^{k}]^{-1} \psi = \II(\partial_t^{-1}  \psi), \qquad \text{with} \qquad
    \partial_t^{-1} \psi(t):=\int_{0}^{t}{\psi(\tau)\,d\tau}.
  \end{align}
  For more general functions $\psi\in \operatorname{dom}(\II)$, the function $\psi$ is replaced by $\II \psi$ in the right-hand side.
\end{lemma}
\begin{proof}
  Following the proof of \eqref{eq:expl_repr_z_transform},
  we observe that $(\partial_t^{h})^{-1} \psi$ corresponds to applying the DG method
  to the ode $y'(t)=\II \psi(t)=\psi$, since we assumed that
  $\psi \in \SS^{p,0}(\TT_h)$.
  We show
  that $\II y$ solves the defining equation of such an approximation.
  Using Lemma~\ref{lemma:dot_vanishes_for_eta},
  a simple calculation shows that
  for $v \in \widetilde{\SS}^{p,0}(\TT_h)$:
  \begin{align*}
      \int_{0}^{\infty}{[\II y]'(t) v(t)\,dt}
    + \sum_{n=0}^{\infty}{\timejump{\II y}{t_n} v(t_n^+)}
    &=
      \int_{0}^{\infty}{y'(t) v(t)\,dt}
      + \sum_{n=0}^{\infty}{\timejump{y}{t_n} v(t_n^+)} \\
    &=
      \int_{0}^{\infty}{\psi(t) v(t)\,dt},
  \end{align*}
  since $y=\partial_t^{-1} \psi$ is continuous. This concludes the proof.
\end{proof}

Another immediate result from the definition is the following representation
  of the discrete differential:
\begin{lemma}
  For $g \in \dom(\II)$ with $g(0)=0$, the
  function $\partial_t^h g$ satisfies for all $v \in \widetilde{\SS}^{p,0}(\TT_h)$
  \begin{align*}
    \int_{0}^{\infty}{\partial_t^h g(t) v(t)\,dt}
    &=  \int_{0}^{\infty}{g'(t) v(t)\,dt}
    + \sum_{n=0}^{\infty}{\timejump{g}{t_n} v(t_n^+)}.
  \end{align*}
  For $g \in C^{1}(\R_+)$ with $g(0)=0$ this means
  $$
  \partial_t^h g = \Pi_{L^2} g',
  $$
  where $\Pi_{L^2}$ denotes the $L^2$ orthogonal projection onto the space $\SS^{p,0}(\TT_h)$.
\end{lemma}
\begin{proof}
  We write $u:=\partial_t^h g$
  and switch to the timestepping framework using sequences
  $G^n:=(\II g)(t_n+\cdot h) \in \PP_p$
  and $U^n:=u(t_n+\cdot h)=(\partial_t^h g)(t_n+\cdot h)$.
  Since $\bs \delta(z)=\bs S - z \bs T(0)^t \bs T(1)$, we see that
  the CQ weights are $W_0=\bs S$, $W_1=\bs T(0)^t \bs T(1)$ and $W_n=0$ for $n>1$.
  Therefore we get $U^n=\bs S G^n - \bs T(0)^t \bs T(1)G^{n-1}$.
  This is the timestepping version of
  \begin{align*}
    \int_{0}^{\infty}{u(t)v^h(t)\,dt}
    &=\int_{0}^{\infty}{ [\II g]'(t) v^h(t)\,dt}
      +  \sum_{n=1}^{\infty}
      {\llbracket  \II g \rrbracket_{t_n} v^h(t_n^+)} \\
    &=
      \int_{0}^{\infty}{ \dot{g}(t) v^h(t)\,dt}
      +  \sum_{n=1}^{\infty}
      {\llbracket  g \rrbracket_{t_n} v^h(t_n^+)}
      \qquad \qquad   \text{ for all } \forall v^h  \in \widetilde{\SS}^{p,0}(\TT_h),
  \end{align*}
  where in the last step we used Lemma~\ref{lemma:dot_vanishes_for_eta}.
  The statement for $g\in C^{1}(\R_+)$ then follows immediately, because the
  jump terms vanish.
\end{proof}

\subsubsection{Relation to Runge-Kutta CQ}
A natural question is in which sense the new DG-CQ method actually consists of a novel method,
and how it might be related on other well-established methods.
We have the following result, relating the DG-approximation combined
with additional quadrature to the established RadauIIa methods:
\begin{proposition}[{\cite[Proposition 2.2 and (2.33)]{H23}}]
  \label{prop:equiv:DG_RK}
  Consider the ODE $y'= f(t,y)$, \quad $y(0)=y_0$ on $(0,1)$.
  Assume that  $y^h$ is computed by taking the DG approximation where the
  integral on the right-hand side is discretized using the Gauss-Radau quadrature:
  $$
  \int_{0}^{1}{f(\tau, y^h(\tau)) q(\tau)} \approx \sum_{j=0}^{p}{\omega_j f(\xi_j,y^h(\xi_j)) q(\xi_j)},
  $$
  where $(\xi_j)_{j=0}^{p}$ are the (right-sided) Radau-Nodes and $(\omega_j)_{j=0}^{p}$ are the corresponding weights.  
  Then
  $y^h(\xi_j)$ coincides with the usual RadauIIa discretization of $y$
  at all Radau nodes $\xi_j$,  $j=0,\dots p$.
\end{proposition}

From this, we can deduce the following result about the equivalence
of convolution quadratures:
\begin{theorem}
  Define $\widetilde{\mathcal{I}}_p g \in \SS^{p,0}(\TT_h)$ as the Radau-interpolation, i.e. 
  $$
  \widetilde{\mathcal{I}}_p g(t_n+h\tau)
  :=\sum_{j=0}^{p}{g(t_n+h\xi_j) \ell_{j}(\tau)} \qquad \forall n \in \N_0, \tau \in (0,1],
  $$
  where $(\ell_{j})_{j=0}^{p}$ denotes the Lagrange polynomials with respect to the
  Radau nodes on $(0,1]$. Then
  $  K(\partial_t^h) \widetilde{\mathcal{I}}_p g $ coincides
  with the RadauIIa CQ approximation in all of the (mapped) Radau interpolation nodes.  
\end{theorem}
\begin{proof}
  By Proposition~\ref{prop:equiv:DG_RK}, approximating the ODE $y'=\zeta y + g$
  with the DG method and quadrature gives the same approximation
  $y_h$ as using the RadauIIa method when only evaluating in the
  Radau nodes.
  By the definition of Radau quadrature, we have that
  $$
  \int_{0}^{1}{\widetilde{\II} g(t) q(t)\,dt}=
  \sum_{j=0}^{p}{\omega_j g(\xi_j) q(\xi_j)},  
  $$
  The terms on the left-had side of \eqref{eq:ode_dg_formulation}  are integrated exactly.
  Therefore using the quadrature in computing $y^h$ is the same as using $\widetilde{\II}g$ as
  the right-hand side of the ODE. 
  We can then apply the representation~\eqref{eq:cq_as_integral} and it's corresponding RK-CQ version
  to conclude that the two convolution quadratures coincide.
\end{proof}

\section{Abstract wave propagation}
\label{sect:abstract_waves}
Before we can analyze the approximations schemes for the scattering problem
given by our convolution quadrature and boundary integral equations,
we first have a deeper look at the time discretization of wave propagation
problems using the discontinuous Galerkin method.
In order to be able to analyze a general class of
boundary conditions, we will perform the analysis
in a very abstract setting inspired by \cite{HQSS17}.

\begin{definition}[Abstract wave problem]
We assume we are given three Hilbert spaces $\VV \subseteq \HH $, and
$\MM$. We assume bounded linear operators $\AA_{\star}: \VV \to \HH$ and $\BB: \VV \to \MM$.

The abstract wave problem is  formulated as:
Given $\bs \xi: (0,T)  \to  \mathbb{M}$, find
$\bs u: \VV \to \HH$
such that $\bs u(0)=0$ as well as:
\begin{align}
  \label{eq:wave_eq_abstrac}
  \dot{\bs{u}}(t) &= \AA_{\star} \bs{u} (t)  \quad \text{and} \qquad
                    \BB \bs{u}(t) = \bs \xi(t) \quad \qquad\forall t>0. 
\end{align}
\end{definition}
The operator $\AA_{\star}$ will later be the first order wave operator $\begin{pmatrix} 0 & \nabla \cdot \\ \nabla & 0\end{pmatrix}$, whereas the operator $B$ is used to encode different boundary conditions.

The core of the theory analyzing this problem is the operator $\AA:=\AA_{\star}|_{ \operatorname{ker}(\BB)}$.
We make the following assumption:
\begin{assumption}
  \label{ass:abstract_operators}
  The operator $B$ admits a bounded right-inverse $\lifting: \MM \to \VV$ such that
  for all $\bs \xi \in \MM$:
  $$
  B \lifting \bs \xi=\bs \xi, \qquad \text{and} \qquad \AA_{\star} \lifting \bs \xi = \lifting \bs\xi.
  $$

  The operator $\AA$ is dissipative, i.e., the following integration by
  parts formula holds for all $\bs u \in \operatorname{ker}(B)$:
  $$
  \big( \AA \bs u,  \bs u\big)_{\HH}  = 0.
  $$

  Furthermore, there exists a constant $M>0$ such that
  \begin{align*}
    \|(A-z\,I)^{-1}\|_{\HH \to \HH} \leq \frac{M}{\Re(z)} \qquad \forall z \in \C \,\text{ s.t. }\Re(z) > 0,
  \end{align*}
  in the sense that the inverse exists and is bounded.
\end{assumption}

These assumptions are enough for the following existence and regularity result

\begin{proposition}
  Assume $\bs \xi \in C^{m+2}(\R,\mathbb{M})$   for
  $m \in \N$ and $\bs \xi(t)=0$ for $t\leq 0$.
  Then, the solution $\bs u$ of \eqref{eq:wave_eq_abstrac} exists and
  satisfies
  $\bs u \in C^{m}(\R,\HH) \cap C^{m-1}(\R,\VV)$. The following estimate holds
  for $1\leq r \leq m$:
  \begin{align}
    \label{eq:apriori_wave}
    \|\bs u^{(r)}(t)\|_{\HH} +\|\bs u^{(r-1)}(t)\|_{\VV}
    &\leq C \sum_{\ell=r}^{r+2}{
      \int_{0}^{t}{\big\|\frac{d^{\ell}}{dt^{\ell}} \bs \xi(t) \big\|_{\MM}\,dt}}.
  \end{align}
  The constant only depends on the constants involved in Assumption~\ref{ass:abstract_operators}.
\end{proposition}
\begin{proof}
  Estimate \eqref{eq:apriori_wave} follows from~\cite[Theorem 4.3]{HQSS17}. Higher order derivatives
  follow by differentiating the equation. See also~\cite[Corollary 1.3.6]{diss_mueller}
  for a similar result in a slightly different formalism.
  Here it is crucial that we assume that $\bs \xi$ and therefore $\bs u$  vanishes at the
    time $t=0$ as well as all of its derivatives.
\end{proof}

We will also need the following small result.
\begin{lemma}
  \label{lemma:A_is_skew_symmetric}
  The operator $\AA$ is skew-symmetric, i.e., for $\bs u,\bs v \in \dom(\AA)$:
  \begin{align*}
    \big( \AA \bs u, \bs v\big)_{\HH}
    &=-\big( \bs u, \AA \bs v\big)_{\HH}.
  \end{align*}
\end{lemma}
\begin{proof}
  We expand the  quadratic expression to get:
  \begin{align*}
    \big( \AA \bs u, \bs v\big)_{\HH} + \big( \bs u, \AA \bs v\big)_{\HH}
    &=
      \big( \AA \bs u,\bs u\big)_{\HH} + \big( \AA \bs v,\bs v\big)_{\HH}
      - \big( \AA (\bs u-\bs v), (\bs u-\bs v)\big)_{\HH}
      =0
  \end{align*}
where in the last step we used  the dissipativity.
\end{proof}

\subsection{DG discretizations of wave problems}
\label{section:abstract_wave_dg}

Given $p \in \N_0$, the DG-timestepping discretization of~\eqref{eq:wave_eq_abstrac} reads:
Find $\bs u^h \in \SS^{p,0}(\TT_{h}) \otimes \VV$
such that for all $ \bs v^h \in \widetilde{\SS}^{p,0}(\TT_{h}) \otimes \HH $
\begin{subequations}
    \label{eq:spacetime_dg_formulation}
\begin{align}
  &\int_{0}^{\infty}{ (\dot{\bs u}^h(t),\bs v^h(t))_{\HH}
  - \big(\AA_{\star}\bs u^h(t),\bs v^h(t)\big)_{\HH} \,dt}
  +  \sum_{n=1}^{\infty}
    {(\llbracket  \bs u^h \rrbracket_{t_n}, (\bs v^h(t_n))^+)_{\HH}} = 0,
      \label{eq:spacetime_dg_formulation1}\\
  &B \bs u^h(t) = \II \bs \xi(t)  \qquad \forall t >0.
\end{align}
\end{subequations}
Note that homogeneous initial conditions are encoded in this weak formulation by
  virtue of the homogeneous right-hand side in \eqref{eq:spacetime_dg_formulation1}.
\begin{remark}
We remark  that the  projection $\II$ can be computed in an element-by element fashion
in a very simple way that is no more expensive than computing an $L^2$-projection or a load vector,
see Section~\ref{sect:practical_aspects}.
\end{remark}

We can also define the timestepping version of~\eqref{eq:spacetime_dg_formulation}
using a fixed basis.
The sequence $(\bs U^n)_{n=0}^{\infty} \subseteq \HH^m$ is given by
\begin{subequations}
  \label{eq:dg_as_timestepping}
  \begin{align}
    \bs S \bs U^n - h\bs M \otimes \AA_{\star} \bs U^n &= \bs T(0)^t \bs T(1) \bs U^{n-1}
                                                         \qquad \text{and }\qquad
    B \bs U^n= \bs \Xi^n,
  \end{align}
\end{subequations}
where $\bs \Xi^n$ is the representation of $\II \bs \xi|_{(t_n,t_{n+1})}$ using the given basis
and $\bs U^{-1}:=0$.


\begin{lemma}
  \label{lemma:well_posed}
  The formulations \eqref{eq:spacetime_dg_formulation} and \eqref{eq:dg_as_timestepping}
  are
  well-posed
  for all $\bs \xi \in \operatorname{dom}(\II) \otimes \mathbb{M}$.
\end{lemma}
\begin{proof}
  We use the reformulation as a matrix problem~\eqref{eq:dg_as_timestepping}.
  By Assumption~\ref{ass:abstract_operators}, $B$ admits a
  right-inverse and $\AA$ is the generator of
  a $C_0$ semigroup.
  By using such this right-inverse to
  lift boundary conditions, we can restrict ourselves to
  homogeneous boundary data. The well-posedness is then equivalent to
  the operator
  $$
  \mathbf{S} 
  -h \mathbf{M} \otimes \AA_{\star}   \qquad \text{mapping } \mathbb{V}^{p+1} \to \mathbb{H}^{p+1}
  $$
  being invertible.
  If $\mathbf{M}^{-1} \mathbf{S}$ is diagonalizable this can be reduced
  to inverting scalar operators of the form $(\lambda - \AA_{\star})^{-1}$,
  which is possible for $\lambda \in \C_+$  by Proposition~\ref{prop:right-inverse}.
  The general case can be dealt with Jordan-forms, and is treated
  rigorously in~\cite[Lemma 2.4]{RSM20}. The
  only necessary ingredient is that $\sigma(\mathbf{M}^{-1}\mathbf{S})\subseteq \C_+$.
  This can be shown directly, or follows from Lemma~\ref{lemma:spectrum_of_delta} by setting $z=0$.
\end{proof}

We summarize some inverse inequalities which we will make
use of throughout the section.
\begin{proposition}[{Inverse inequalities, \cite[Eqn. (3.6.1), (3.6.3), (3.6.4)]{schwab_book}}]
  Let $I:=(a,b)$, $h:=b-a$.
  Then, there exist generic constants such that
  following inequalities hold for $q \in \{2,\infty\}$ and
  for all $u \in \PP_{p}$:
  \begin{align}
    \label{eq:invest}
    \|{u'}\|_{L^{q}(I)}
      &\lesssim \frac{p^2}{h} \|u\|_{L^2(I)}
        \qquad \text{and}
        \qquad
        \|u\|_{L^{\infty}(I)} \lesssim \frac{p}{\sqrt{h}} \|u\|_{L^2(I)}.                 
  \end{align}
\end{proposition}
 
In order to be able to concisely write down the asymptotic behavior of the method,
we use the definition $\pstar:=\max(p,1)$ and also write $\tstar:=\max(h\lceil t/h\rceil,1)$.
\begin{theorem}[Discrete Stability]
  \label{thm:disc_stability}
  Given $\bs f \in L^2(\R_+, \HH)$, $\bs \xi \in H^1_{\mathrm{loc}}(\R_+,\MM)$. Let $\bs u^h$ solve
  for all $\bs v^h \in \widetilde{\SS}^{p,0}(\TT_h)$
  \begin{subequations}
    \label{eq:spacetime_dg_formulation_perturbed}
    \begin{align}
      &\int_{0}^{\infty}{ \big(\dot{\bs u}^h(t)
        - \AA_{\star}\bs u^h(t),\bs v^h(t)\big)_{\HH} \,dt}
        +  \sum_{n=1}^{\infty}
        {(\llbracket  \bs u^h \rrbracket_{t_n}, \bs v^h(t_n^+))_{\HH}} =
        \int_{0}^{\infty}{ (\bs f(t), \bs v^h(t))_{\HH}\,dt},
        \label{eq:spacetime_dg_formulation_pert11}\\
      &B \bs u^h(t) = \II \bs \xi(t)  \qquad \forall t >0.
\end{align}
\end{subequations}
There exits a constant $C>0$, such that the following stability estimates hold for all $ t>0$:
\begin{subequations}
  \label{eq:est_stab_all}
\begin{align}
  \max_{n h \leq t}\|u^h(t_n^-)\|_{\HH}^2 &\lesssim
  C \tstar    \int_{0}^{\tstar}{
    \big(\|\bs f(\tau)\|^2_{\HH}+\|\bs \xi(\tau)\|_{\MM}^2+\|\dot{\bs\xi}(\tau)\|_{\MM}\big)\,d\tau},\\
\int_{0}^{t}{\|\bs u^h(\tau)\|_{\HH}^2 \,d\tau}
  &\leq C \tstar^2    \int_{0}^{\tstar}{
    \big(\|\bs f(\tau)\|^2_{\HH}+\|\bs \xi(\tau)\|_{\MM}^2+\|\dot{\bs\xi}(\tau)\|_{\MM}^2 \big)\,d\tau}, \\
  \|\bs u^h(t)\|_{\HH}^2
  &\leq C {p^2} \, \tstar 
    \int_{0}^{\tstar}{
    \big(\|\bs f(\tau)\|^2_{\HH}+\|\bs \xi(\tau)\|_{\MM}^2+\|\dot{\bs\xi}(\tau)\|_{\MM}^2 \big)\,d\tau}.
    \label{eq:est_stab_linf}
\end{align}
\end{subequations}
\end{theorem}
\begin{proof}
  For simplicity, we start with the case $\bs \xi=0$.
  Since $\bs u^h$ is a piecewise polynomial,
  we can use the test function $\bs v:=\bs u^h$ on a fixed interval $(t_{n},t_{n+1})$
  and $\bs v:=0$ otherwise.  We get
  \begin{align}
    \label{eq:stab_proof1}
    \int_{t_{n}}^{t_{n+1}}{
    \big(\dot{\bs u}^h(t),\bs u^h(t)\big)_{\HH}  - \big(\AA_{\star} \bs u^h(t),\bs u^h(t)\big)_{\HH}\,dt}
    + \big(\timejump{\bs u^h}{t_n},(\bs u^h(t_{n}^+)) \big)_{\HH}
    &= \int_{t_n}^{t_{n+1}}{ \big( \bs f(t),\bs u^h(t)\big)_{\HH}\,dt}.
   \end{align}
   Since we assumed $\bs \xi=0$, we have $\bs u^h(t) \in \ker(B)$ for all times $t >0$ and we can therefore use
   the dissipativity of $\AA$ to conclude
   $\int_{0}^{T}{(\AA_{\star} \bs u^h(t),\bs u^h(t))_{\HH}\,dt}=0.$
   
   Integration by parts and reordering some terms,
   by using the easy to prove relation $2(b-a)b={b^2+(b-a)^2-a^2}$
   gives:
  \begin{align*}
  &\int_{t_{n}}^{t_{n+1}}{ \big(\dot{\bs u}^h(t),\bs u^h(t)\big)_{\HH} \,dt}
  + \big(\timejump{\bs u^h}{t_n},\bs u^h(t_{n}^+) \big)_{\HH} \\
    &\qquad \qquad =\frac{1}{2}\Big(
    \|\bs u^h(t_{n+1}^-)\|_{\HH}^2 - \|\bs u^h(t_{n}^+))\|_{\HH}^2 \Big)
    +
      \frac{1}{2}\Big(
      \|\timejump{\bs u^h}{t_n}\|^2
      +  \|\bs u^h(t_n^+)\|_{\HH}^2 -\|\bs u^h(t_n^-)\|_{\HH}^2 \Big) \\
    &\qquad \qquad =
      \frac{1}{2}\Big(
      \|\bs u^h(t_{n+1}^-)\|_{\HH}^2 - \|\bs u^h(t_{n}^-)\|_{\HH}^2 
      +\|\timejump{\bs u^h}{t_n}\|_{\HH}^2
      \Big).
  \end{align*}
  Inserting all of this  into~\eqref{eq:stab_proof1}  means that $\bs u^h$ satisfies 
  \begin{multline}
    \frac{1}{2} \big( \|\bs u^h(t_{n+1}^-)\|^2_{\HH}
    + \|\timejump{\bs u^h}{t_n}\|^2_{\HH}\big) 
    =
    \frac{1}{2}\|\bs u^h(t_n^-)\|^2_{\HH}    
    +\int_{t_{n}}^{t_{n+1}}{\big(\bs f,\bs u^h(t)\big)_{\HH}\,dt}  \\
    \leq
    \frac{1}{2}\|\bs u^h(t_n^-)\|^2_{\HH}  
    +\Big(\int_{t_n}^{t_{n+1}}{
       \|\bs f(t)\|_{\HH}^2
      \,dt}\Big)^{1/2}
    \Big(\int_{t_n}^{t_{n+1}}{\|{\bs u^h}(t)\|^2_{\HH}\,dt}\Big)^{1/2}.
  \label{eq:est_theta_hh}
  \end{multline}
  The norms of $\bs u^h$ on the left- and right-hand side do not match up.
  Therefore we need another estimate.

    For the second test function, we make use of the
    resolvent operator $R: \HH \to \VV,$ defined by $R\bs h:=(-\AA+\omega)^{-1} \bs h$
    where $\omega>0$ will be chosen later.
    We note the following
    identities for all $\bs z \in \HH$, $\bs y \in \dom(\AA)$
    due to the dissipativity of $\AA$ and Lemma~\ref{lemma:A_is_skew_symmetric}:
    \begin{subequations}
      \label{eq:some_relations}
    \begin{align}
      \big(\bs z, R \bs z\big)_{\HH}
      & =
        \big((-A+ \omega) R  \bs  z, R  \bs z\big)_{\HH} =
        \big(-A R  \bs z, R  \bs z\big)_{\HH}  +
        \big(\omega R  \bs z, R  \bs z\big)_{\HH} 
      =\omega \|R  \bs z\|^2_{\HH}, \\
      \big( \AA  \bs y, R  \bs z\big)_{\HH}
      &= - \big(  \bs y, \AA R  \bs z \big)_{\HH}
      = \big(  \bs y,  \bs  z\big)_{\HH} - \omega \big( \bs y,  R  \bs z\big)_{\HH}.
    \end{align}
  \end{subequations}
  
    The test function we now consider is
    $\bs v^h|_{(t_n,t_{n+1})}:=(t-t_n)R\dot{\bs u}^h$.
    By Assumption~\ref{ass:abstract_operators} this function exists and is bounded
    by $\|\bs v^h(t)\|_{\HH}\leq M \omega^{-1} (t-t_n)\|\dot{\bs u}^h\|_{\HH}$. 
    Inserting the new test function and using \eqref{eq:some_relations}
    with $ \bs z:=\dot{\bs u}^h(t)$ and $ \bs y:=\bs u^h(t)$ yields    
  \begin{multline}
    \int_{t_{n}}^{t_{n+1}}{ (t-t_n)\Big(\omega \|R \dot{\bs u}^h(t)\|^2_{\HH}
      -  \big( \bs u^{h}(t), \dot{\bs u}^{h}(t)\big)_{\HH} \Big)\,dt} \\
  = -\omega \int_{t_{n}}^{t_{n+1}}{ (t-t_n)\big( \bs u^{h}(t), R\dot{\bs u}^{h}(t)\big) \,dt}
  +\int_{t_n}^{t_{n+1}}{(t-t_{n})\bs f(t) R \dot{\bs u}^h(t) \,dt}.
  \label{eq:stab_int3}
  \end{multline}
  Integration by parts allows us to rewrite the second term as:
  \begin{multline*}
   - \int_{t_{n}}^{t_{n+1}}{ (t-t_n)\big( \bs u^{h}(t), \dot{\bs u}^{h}(t)\big)_{\HH} \,dt}
    =
      \int_{t_{n}}^{t_{n+1}}{ \big( [(t-t_n) \bs u^{h}(t)]', \bs u^{h}(t)\big)_{\HH} \,dt}
      - h \| \bs u^h{t_{n+1}}\|_{\HH}^2 \\
      =
      \int_{t_{n}}^{t_{n+1}}{ \|\bs u^{h}(t)\|^2_{\HH} \,dt}
      +
      \int_{t_{n}}^{t_{n+1}}{ (t-t_n)\big( \bs u^{h}(t), \dot{\bs u}^{h}(t)\big)_{\HH} \,dt}
      - h \| \bs u^h{t_{n+1}}\|_{\HH}^2,
  \end{multline*}
  or
  \begin{align*}
    - \int_{t_{n}}^{t_{n+1}}{ (t-t_n)\big( \bs u^{h}(t), \dot{\bs u}^{h}(t)\big)_{\HH} \,dt}
    &=
      \frac{1}{2}\int_{t_{n}}^{t_{n+1}}{ \|\bs u^{h}(t)\|^2_{\HH} \,dt}      
      - \frac{h}{2} \| \bs u^h(t_{n+1})\|_{\HH}^2.
  \end{align*} 
  From this, we get in \eqref{eq:stab_int3} by using Cauchy-Schwarz and Young's  inequalities:
  \begin{multline*}
    \int_{t_{n}}^{t_{n+1}}{ 2\omega(t-t_n) \|R \dot{\bs u}^h(t)\|^2_{\HH} + \|\bs u^h(t)\|_{\HH}^2\,dt}
    \leq h \|\bs u^{h}(t_{n+1}^-)\|^2_{\HH} +
    2\omega \int_{t_{n}}^{t_{n+1}}{ (t-t_n)\| \bs u^{h}(t)\|^2_{\HH}\,dt} \\
    + \omega \int_{t_{n}}^{t_{n+1}}{ (t-t_n)\| R  \dot{\bs u}^{h}(t)\|^2_{\HH}\,dt }
    + 2\omega^{-1}\int_{t_n}^{t_{n+1}}{(t-t_{n})\|\bs f(t)\|^2_{\HH} \,dt}.
  \end{multline*}
  By taking $\omega:=\nicefrac{1}{4h}$, we can absorb the second and third term of the
  right-hand side to get:
  \begin{align}
    \label{eq:stab_proof_4}
      \frac{1}{2h}\int_{t_{n}}^{t_{n+1}}{ (t-t_n)\|R \dot{\bs u}^h(t)\|^2_{\HH} \,dt}
      +\int_{t_{n}}^{t_{n+1}}{ \|\bs u^h\|_{\HH}^2\,dt}
      \leq 2h \|\bs u^{h}(t_{n+1}^-)\|_{\HH} + 16h\int_{t_n}^{t_{n+1}}{(t-t_{n})\|\bs f(t)\|^2_{\HH} \,dt}.
    \end{align}

  We now fix $T\geq \max(1, t_{n+1})$.
  Going back to~\eqref{eq:est_theta_hh} and using the $L^2$-bound in~\eqref{eq:stab_proof_4}, we get
  using a weighted form of Young's inequality:
  \begin{multline*}
    \frac{1}{2} \big( \|\bs u^h(t_{n+1}^-)\|^2_{\HH}
    +\|\timejump{\bs u^h}{t_{n}}\|^2_{\HH}
    - \|\bs u^h(t_n^-)\|^2_{\HH}\big) \\
      \begin{aligned}[t]
        &\lesssim 
          \Big(\int_{t_n}^{t_{n+1}}{\|\bs f(t)\|_{\HH}^2
          \,dt}\Big)^{1/2} 
          \Big(h\|\bs u^h(t_{n+1}^-)\|_{\HH}^2 + 
          h\int_{t_n}^{t_{n+1}}{(t-t_n)\|\bs f(t)\|_{\HH}^2\big)\,dt}
          \Big)^{1/2} \\
    &\leq
      C T \int_{t_n}^{t_{n+1}}{\|\bs f(t)\|_{\HH}^2\,dt}
      + \frac{h}{2 T}\|\bs u^h(t_{n+1}^-)\|_{\HH}^2
      + \frac{h}{2 T}\int_{t_n}^{t_{n+1}}{(t-t_n)\|\bs f(t)\|_{\HH}^2\big)\,dt}.
    \end{aligned}
  \end{multline*}
  By the discrete Gronwall Lemma (e.g.~\cite[Lemma~10.5]{thomee_book}), this concludes:
  \begin{align}
    \|\bs u^h(t_{n+1}^-)\|^2_{\HH}
    +\|\timejump{\bs u^h}{t_{n}}\|^2_{\HH}
    &\lesssim
      e^{C \frac{(n+1) h}{T}}
      \Big( T \int_{0}^{t_{n+1}}{\|\bs f(t)\|_{\HH}^2\,dt}
      + \frac{h}{T}\sum_{n=0}^{n}{\int_{t_j}^{t_{j+1}}{(t-t_j)\|\bs f(t)\|_{\HH}^2\,dt}}\Big) \nonumber
    \\  
    &\lesssim C T \int_{0}^{t_{n+1}}{\|\bs f(t)\|_{\HH}^2\,dt}
      \label{eq:best_approx_tracejumps}
  \end{align}
  where we only kept the dominant term and used $n k = t_n \leq t$.
  To get the bound in the $L^2$-norm in time, we go back to~\eqref{eq:stab_proof_4} to get:
    \begin{align}
      \label{eq:stab_proof_5}
      \int_{t_{n}}^{t_{n+1}}{ \|\bs u^h\|_{\HH}^2\,dt}
      &\leq h T \int_{0}^{t_{n+1}}{\|\bs f(t)\|_{\HH}^2\,dt}.
    \end{align}
    Summing up gives the stated result. For the pointwise in time estimate, we
    use the inverse inequality~\eqref{eq:invest}
    and~\eqref{eq:stab_proof_5}:
    \begin{align*}
    \max_{\tau\in [t_n,t_{n+1}]}\|\bs u^h(\tau)\|_{\HH}^2
      &\lesssim \frac{p^2}{h} \int_{t_n}^{t_{n+1}}{\|\bs u^h(\tau)\|_{\HH}^2\,d \tau} \lesssim p^2 T \int_{0}^{t_{n+1}}{\|\bs f(t)\|_{\HH}^2\,dt}.
    \end{align*}

  Now that we have treated the case $\bs \xi=0$ we can extend the result to the general case.
  Using the lifting operator $\lifting$ from Assumption~\ref{ass:abstract_operators},
  we split the function $\bs u^h:=\bs u^h_0 + \lifting \II\bs \xi$. The function
  $\bs u^h_0$ then satisfies $B\bs u^h_0= B\bs u^h - B \lifting \II \bs \xi=0$ by the lifting property.
  It also satisfies \eqref{eq:spacetime_dg_formulation_perturbed} with the following modified
  right-hand side:
  $$
  \widetilde{\bs f}:={\bs f} + \lifting \dot{\bs \xi} -  \lifting \II\bs \xi,  
  $$
  This can be shown by applying Lemma~\ref{lemma:dot_vanishes_for_eta} and using the
  property $\AA_{\star} \lifting=\lifting$.
  We immediately get from, the stability of the lifting and interpolation:
  $\|\widetilde{\bs f}\|_{\HH} \lesssim
  \|\bs f(\tau)\|+ \|\dot{\bs \xi}(\tau)\|_{\MM} + \|\bs \xi(\tau)\|_{\MM}$.
  The stated estimates in \eqref{eq:est_stab_all} then follows from the already
  established  case $\bs \xi=0$ and the triangle inequality
  $\|\bs u^h\|_{\HH}\lesssim \|\bs u^{h}_{0}\|_{\HH} + \|\bs \xi\|_{\MM}$.
  Terms of the form $\| \bs \xi(t)\|^2_{\MM}$ are absorbed using the 1d Sobolev embedding
  into the integral terms present in the right-hand side of~\eqref{eq:est_stab_all}.
\end{proof}

We can now state and prove our first convergence result. It controls the
field $\bs u$ but will not strong enough to bound the BEM-unknowns
which are recovered via trace operators.

\begin{theorem}
  \label{thm:best_approx} 
  Let $\bs u$ solve~\eqref{eq:wave_eq_abstrac} and
  $\bs{u}^h$ solve~\eqref{eq:spacetime_dg_formulation}.
  Then, the following estimate holds  for all times $t > 0$
 \begin{align}
   \label{eq:best_approx}
   \int_{0}^{T}{\|\bs u(t) - \bs u^{h}(t)\|_{\HH}^2\,dt}
   &\lesssim     
     \tstar^2\int_{0}^{\tstar}{
      \|\bs u(\tau) - [\II \bs u](\tau)\|_{\mathbb{V}}^2\,d\tau},  \\
    \|\bs u(t) - \bs u^{h}(t)\|_{\mathbb{H}} 
   &\lesssim
     \|\bs u(t) - \II \bs u(t)\|_{\mathbb{H}} 
     +{p}
     \sqrt{\tstar\int_{0}^{\tstar}{
      \|\bs u(\tau) - [\II \bs u](\tau)\|_{\mathbb{V}}^2\,d\tau}},
\end{align}
where we recall $\tstar:=\max(h\lceil  t/h\rceil,1)$, $\pstar:=\max(p,1)$.
The implied constant is independent of $k, p$, and $t$.
At the nodes $t=t_n$, the additional factor $p$ can be dropped.
\end{theorem}
\begin{proof}
  We set $\bs \theta(t)=\bs u^h(t) - \II \bs u(t)$
  and $\bs \eta(t):=\II \bs u(t) - \bs u(t)$.
  One can easily check that the DG method is consistent,
    i.e. $\bs u$ also solves~\eqref{eq:spacetime_dg_formulation}.
  By Lemma~\ref{lemma:dot_vanishes_for_eta},
  $\bs\theta$ solves the following space-time problem:
  For any
  test-function $\bs v \in \widetilde{\SS}^{p,0}(\TT_h) \otimes \HH$:
  \begin{align}
    \label{eq:dg_eqn_of_theta}
    \int_{0}^{\infty}{ \big(\dot{\bs \theta}(t), \bs v(t)\big)_{\HH}  - \big(\AA_{\star}\bs \theta(t),\bs v(t)\big)_{\HH}\,dt}
    + \sum_{n=1}^{\infty}{\big(\timejump{\bs \theta}{t_{n}},\bs v(t_n^+) \big)_{\mathbb{H}}}  
    &=
    \int_{0}^{\infty}{\big(\AA_{\star}\bs \eta(t),\bs v(t)\big)_{\HH}\,dt},
  \end{align}
  where we used that $\bs u$ solves the PDE~\eqref{eq:wave_eq_abstrac}.
  (Note that the right-hand side does not vanish
  because we required $\bs v \in \widetilde{\SS}^{p,0}(\TT_h)$ instead of
  $\bs v \in \widetilde{\SS}^{p-1,0}(\TT_h)$ in Proposition~\ref{prop:def_II}.)
  For the boundary conditions, we get
  $
  B \bs \theta(t)=\bs u^h - \II \bs u = \II \bs \xi - \II \bs \xi=0.
  $
  We can therefore invoke the stability  from Theorem~\ref{thm:disc_stability}
  to get the stated result since $\|\AA_{\star} \bs \eta\|_{\HH} \lesssim \|\bs \eta\|_{\VV}$.  
\end{proof}

In order to later estimate the traces of $\bs u$, we need control of the error in
stronger norms.

\begin{theorem}[Stability in stronger norms]
  \label{thm:strong_stability}
  Let $\bs u^h$  solve~\eqref{eq:spacetime_dg_formulation_perturbed}.
  Then, the following stability estimate holds for all $t > 0$:
  \begin{subequations}
    \label{eq:best_approx_strong_norm}
    \begin{align}
      \sqrt{\int_{0}^{T}{\|\bs u^{h}(t)\|_{\mathbb{V}}^2}}
      &\lesssim \frac{ \tstar p^2}{h}
        \sqrt{\int_{0}^{t_{n+1}}
      {\big(\|\bs f(\tau)\|^2_{\HH}+\|\bs \xi(\tau)\|_{\MM}^2+\|\dot{\bs\xi}(\tau)\|_{\MM}^2 \big)\,d\tau}} ,
            \label{eq:stab_strong_norm:L2}\\
      \|\bs u^{h}(t)\|_{\mathbb{V}}
      &\lesssim \frac{p^3\tstar^{1/2}}{h} 
        \sqrt{\int_{0}^{t_{n+1}}
      {\big(\|\bs f(\tau)\|^2_{\HH}+\|\bs \xi(\tau)\|_{\MM}^2+\|\dot{\bs\xi}(\tau)\|_{\MM}^2 \big)\,d\tau} }.
      \label{eq:stab_strong_norm:lind}
  \end{align}
\end{subequations}
\end{theorem}
\begin{proof}
  We again work
  on a single element $I=(t_n,t_{n+1})$. We convert the equation into the strong
  form by using the adjoint of the trace operators $T^t: \R \to \PP_{q}$ given by
  $$
  \big(T^t u, v\big)_{L^2(I)} := u v(t_n^+) \qquad \forall v \in \PP_q
  \qquad \text{with } \qquad \|T^t u\|_{L^2(I)} =  \frac{p}{\sqrt{h}} | u|,
  $$
  where for the norm bound, we used the $L^{\infty}-L^2$ inverse estimate
  from \eqref{eq:invest}.
  Then~\eqref{eq:spacetime_dg_formulation_perturbed}  on $(t_n,t_{n+1})$ is equivalent to
  $$
  \AA_{\star} \bs u^h =\dot{\bs u}^h + \Pi_{L^2} \bs f + T^t \timejump{\bs u^h}{t_n}
  $$
  We bound each component on the right individually using the inverse inequalities from~\eqref{eq:invest}:
  \begin{align*}    
    \|\dot{\bs u}^h\|_{L^q(I,\HH)}
    &\lesssim \frac{p^2}{h} \|{\bs u}^h\|_{L^q(I,\HH)},
    \qquad \text{ for } q \in \{2,\infty\}\\
    \|\Pi_{L^2} \bs f\|_{L^\infty(I,\HH)}
    &\lesssim \frac{p}{\sqrt{h}} \|\Pi_{L^2} \bs f\|_{L^2(I,\HH)}
                                             \lesssim \frac{p}{\sqrt{h}} \| \bs f\|_{L^2(I,\HH)}, \\
    \|T^t \timejump{\bs u^h}{t_n} \|_{L^\infty(I,\HH)}
    &\lesssim  \frac{p}{\sqrt{h}} \|T^t \timejump{\bs u^h}{t_n} \|_{L^2(I,\HH)}
      \lesssim \frac{p^2}{h} \|\timejump{\bs u^h}{t_n} \|_{\HH}.                                     
  \end{align*}
  All the terms on the right-hand side have already been bounded in Theorem~\ref{thm:disc_stability}.
  The result follows by collecting the previous estimates where
  the term $\dot{\bs u^h}$ dominates.   
\end{proof}

We can leverage the stronger stability estimate into  a better approximation result:
\begin{theorem}
  \label{thm:best_approx_strong}
  Let $\bs u$ solve~\eqref{eq:wave_eq_abstrac}.
  Then, the following estimate holds for all $t > 0$
  \begin{align}
    \|\bs u(t) - \bs u^{k}(t)\|_{\mathbb{V}}
    &\lesssim
      \|\bs u(t) - \II \bs u(t)\|_{\mathbb{V}}
      +\frac{\pstar^3}{h}\sqrt{\tstar\int_{0}^{\tstar}{
      \|\bs u(t) - [\II \bs u](t)\|_{\mathbb{V}}^2\,dt}}.
      \label{eq:best_approx_strong_norm:A}
  \end{align}

\end{theorem}
\begin{proof}
  We again split $\bs \theta:=\bs{u}^h -  \II \bs{u}(t)$ and $\bs{\eta}:=\bs{u} - \II \bs u$.
  Just as before, we use the fact that $\bs \theta$ solves~\eqref{eq:dg_eqn_of_theta}
  and conclude by applying Theorem~\ref{thm:strong_stability}.
\end{proof}

If we insert further knowledge on the regularity of the solution,
we get the sought after root-exponential convergence.
\begin{corollary}
  \label{cor:conv_for_smooth_HH}
  Assume that the exact  solution $\bs u$ of \eqref{eq:wave_eq_abstrac}
  satisfies
  $$
  \bs u \in C^{\infty}((0,T); \mathbb{V})
  \quad \text{such that}\quad
  \sup_{t\in(0,T)}\| \bs u^{(\ell)}(t) \|_{\mathbb{V}} \leq C_{\bs u}  \gamma_{\bs u}^{\ell} (\ell!)^{\omega}
  $$
  for some constants $C_{\bs u}, \gamma_{\bs u} > 0 $ and $\omega \geq 1$, then for all $t \leq T$:
  \begin{align*}
    \|\bs u(t) - \bs u^{k}(t)\|_{\mathbb{V}}    
    &\leq C C_{\bs u} \tstar \Big(\frac{k^{\nicefrac{1}{\omega}}}{k^{\nicefrac{1}{\omega}}+\sigma}\Big)^{\sqrt[\omega]{p}}
  \end{align*}
  for some constants $C>0$, $\sigma >0$ that depend only on $\omega$ and $\gamma_{\bs u}$. 
\end{corollary}
\begin{proof}
  By Theorem~\ref{thm:best_approx_strong}, we only have to work out the approximation
  error $\bs u-\II \bs u$. This can be controlled via
  Proposition~\ref{prop:def_II}.
  For simpler notation, we assume $\gamma_{\bs u}\geq 1$ and $p>1$.
  Stirling's formula  gives
  \begin{align*}
    \frac{\Gamma(p+1-s)}{\Gamma(p+1+s)}
    &\lesssim
    \sqrt{\frac{p-s}{p+s}} \frac{(p-s)^{p-s}}{(p+s)^{p+s}} e^{2s}
    =
      \sqrt{\frac{p-s}{p+s}} \Big( \frac{p-s}{p+s}\Big)^{p-s} \frac{1}{(p+s)^{2s}} e^{2s}
      \lesssim \Big( \frac{e}{p+s}\Big)^{2s}.
  \end{align*}
  This estimate turns the approximation property from Proposition~\ref{prop:def_II}(\ref{it:prop:def_II:1})
  into
  \begin{align*}    
    \| \bs u - \II \bs u\|^2_{L^{\infty}(0,T)}
    &\lesssim k^{2 s+2 } p^{-2s} e^{2s}
    \sum_{\ell=0}^{s+1}{\max_{t\in [0,T]}\|\bs u^{(\ell)}\|^2_{\VV}} 
    \,\lesssim \,
      k^{2 s+2 } p^{-2s}  e^{2s}
      s C_{\bs u}^2  \gamma_{\bs u}^{2s} s^{\omega} s^{2\omega  s} e^{-2\omega s}.
  \end{align*}  
  Since this estimate is valid for all $s \geq 1$,
  we take $s$ such that $p= \beta s^{\omega}$ for $\beta \geq 1$ to be fixed later.
  This means
  \begin{align*}    
    \| \bs u - \II \bs u\|_{L^{\infty}(0,T)}^2
    &\lesssim
      k^{2 s+2 } \beta^{-2s} s^{-2s\omega} e^{2s}
      s C_{\bs u}  \gamma_{\bs u}^{2s} s^{\omega }s^{2\omega s} e^{-2\omega s}
      = C_{\bs u}
      k^2 s^{1+\omega} \Big( \frac{k e \gamma_{\bs u}}{\beta e^{\omega}}\Big)^{2s}.
  \end{align*}
  Carefully balancing the terms will give the stated result:
  We distinguish two cases.
  For $k\leq \nicefrac{1}{2\gamma_{\bs u}}$ we can take $\beta=1$ and set
  $\sigma:=\frac{1}{2\gamma_{\bs u}}$. We calculate:
  \begin{align*}
    \Big(\frac{k e \gamma_{\bs u}  }{\beta e^{\omega}}\Big)^{2s}
    \leq \Big(\frac{k}{\big(\frac{1}{\gamma_{\bs u}} - k\big) + k }\Big)^{2s}
    \leq \Big(\frac{k}{\sigma + k }\Big)^{2p^{1/\omega}}.
  \end{align*}
        
  For $k\geq \nicefrac{1}{2\gamma_{\bs u}}$, we set $\beta:=2k\gamma_{\bs u}$ and further estimate
  using $e^{-x} \leq \nicefrac{1}{1+x}$:
  \begin{align*}
    \Big( \frac{k e \gamma_{\bs u}}{\beta e^{\omega}}\Big)^{2s}
    &\leq \Big(\frac{1}{2}\Big)^{\frac{2 p^{1/\omega}}{(2 k \gamma_{\bs u})^{1/\omega}}}
      = \Big(e^{-\frac{\ln(2)}{(2k \gamma_{\bs u})^{1/\omega}}}\Big)^{2p^{1/\omega}}
      \!\!\leq\!  \Big(\frac{1}{1 + \nicefrac{\ln(2)}{(2k\gamma_{\bs u})^{1/\omega}}}\Big)^{2p^{1/\omega}}
      \!\!= \!  \Big(\frac{k^{\nicefrac{1}{\omega}}}{k^{\nicefrac{1}{\omega}} + \nicefrac{\ln(2)}{(2\gamma_{\bs u})^{1/\omega}}}\Big)^{2p^{1/\omega}}.
  \end{align*}
  This concludes the proof after absorbing the algebraic  term $k^2 s^{1+\omega} \sim k^{1-1/\omega} p^{1+1/\omega}$,
  as well as the terms $\nicefrac{\pstar^2\ln(\pstar)}{k}$ from Theorem~\ref{thm:best_approx_strong}  into
  the geometric term by slightly reducing $\sigma$.
\end{proof}

\section{ The CQ scattering problem}
\label{sect:analysis}
In this section, we will use the previously developed abstract theory
to present and analyze a discretization scheme for the scattering problem
in~Section~\ref{sect:model_problem}.
To do so, we show that we can write the CQ approximations as
traces of functions which fit the abstract wave propagation formalism
taken from~\cite{HQSS17} and 
laid out in Section~\ref{sect:abstract_waves}.
\subsection{The Dirichlet problem}
For simplicity, we only pick out one of the discretizations in Proposition~\ref{prop:different_model_problems}
and focus on the Dirichlet problem, which is discretized using a direct method. We
also ignore for the moment the discretization in space. A Galerkin procedure can 
be  easily included in the analysis. We further comment on this later on in Appendix~\ref{sect:galerkin}.

Given $\uinc \in C^1(\R_+,H^{1/2}(\Gamma)$, the continuous problem reads
find $\lambda \in C(\R_+, H^{-1/2}(\Gamma))$ such that
\begin{align}
  \label{eq:dirichlet}
  \partial_t V(\partial_t) \lambda =  \Big( \frac{1}{2} - K(\partial_t)\Big) \gamma^+\duinc
  \quad \text{and } \quad
  u:= -S(\partial_t) \lambda - \partial_t^{-1} D(\partial_t) \gamma^+\duinc.
\end{align}

We now fit the problem into the abstract frame from Section~\ref{sect:abstract_waves}.
We set
\begin{align}
  \label{eq:def_spaces}
  \HH&:=L^2(\R^d \setminus \Gamma) \times [L^2(\R^d \setminus \Gamma)]^d, \qquad
       \VV:=H^1(\R^d \setminus \Gamma) \times \Hdiv{\R^d \setminus \Gamma},
\end{align}
$\MM:=H^{1/2}(\Gamma)$ and  use the operators
\begin{subequations}
    \label{eq:def_operators}
\begin{alignat}{4}
  \AA_\star:& \quad & \mathbb{V} &\to \mathbb{H}:
  &(v,\bs w)^t &\mapsto (\nabla \cdot \bs w, \nabla v)^t \\
  \BB:& \quad &\mathbb{V} &\to \mathbb{M} : \quad
  &(v,\bs w)^t &\mapsto (\gamma^+ v, \gamma^- v)^t .
\end{alignat}
\end{subequations}

  \begin{remark}
    \label{rem:why_differentiate}
    Equation~\eqref{eq:dirichlet}
    contains an  additional time derivative,
    compared to what is often used,
    because it allows us to
    recover the unknown density $\lambda=\normaltjump{\bs w}$.
    Note also that in the first-order formulation $\bs u=(\dot{u},\nabla u)$,
    the natural object is the derivative $v=\dot{u}$.
    We could also remove the operator
    $\partial_t$ in front of $V(\partial_t)$ and use $\gamma \uinc$ as
    boundary data, but would then have to
    consider $\lambda=\partial_t \normaltjump{\bs w}$. This would slightly
    complicate the analysis, but would also be possible by invoking
    an inverse estimate to bound $\|\partial_t^h \bs w\|_{ \Hdiv{\R^d \setminus \Gamma}}$.
\end{remark}

The problem~\eqref{eq:wave_eq_abstrac} is well-posed and fits into the abstract framework
of Section~\ref{sect:abstract_waves}.
\begin{proposition}[{\cite{HQSS17}}]
  \label{prop:right-inverse}
  The operators $\AA_\star$ and $B$ satisfy Assumption~\ref{ass:abstract_operators}.
\end{proposition}
\begin{proof}
  All of this was proven in~{\cite{HQSS17}}.
  See Proposition 4.1 for the integration by parts formula,  and
  Proposition 4.2 for the lifting operator. The statement
  on the resolvent operator is a standard result
  on operator semigroups, often connected with the names Lumer and Phillips.
  See for example \cite[Chapter 1, Theorem 4.3]{Paz83}.  
\end{proof}

We recall the following continuous equivalence between boundary integrals
and the partial differential equation. Since we won't need this
result for our analysis, we skip over the details of the proof.
\begin{proposition}
  \label{prop:equiv_PDE_IE}
  If $\lambda$ solves \eqref{eq:dirichlet} and $u$ is defined via the single and double layer potentials,
  then the pair $\bs u:=(v,\bs{w}):=(\dot{u},\nabla u)$ solves the abstract wave problem
  \begin{align}
    \label{eq:pde_for_dirichlet}
    \dot{\bs u}=\AA_{\star} \bs u \qquad \text{and} \qquad B \bs u(t)= \bs \xi, \qquad \bs u(0)=(0,0)
  \end{align}
  for $\bs \xi:=(-\gamma^+\duinc,0)$. Similarly, if $\bs u $ solves the evolution problem
  then $\lambda:=\normaltjump{\bs w}$ solves~\eqref{eq:dirichlet}.
\end{proposition}
\begin{proof}
  It is well known, that the function $u$ defined by the retarded potentials
  satisfies the wave equation, of which
  \eqref{eq:pde_for_dirichlet} is just the first-order formulation
  using the variables $(\dot{u}, \nabla u)$. The boundary conditions
  follow by taking traces and using the trace properties from Section~\ref{sect:bem_stuff}.
  See~\cite{book_sayas} for a rigorous treatment.
\end{proof}

We have the following regularity result for the time evolution $\bs u$:
\begin{proposition}[Well-posedness and regularity]
  \label{prop:well_posedness_dirichlet}
  Assume that the boundary trace  $\gamma^+ \uinc$ is in the Gevrey class $\mathcal{G}_{\omega}$, i.e., there
  exist constants $\omega \geq 1$, $C(T)>0$ and $\rho >0$ such that for all $t \in (0,T)$
  $$
  \|\frac{d^{\ell}}{dt^{\ell}} \gamma^+ \uinc(t) \|_{H^{ 1/2}(\Gamma)}
  \leq C(T) \rho^{\ell} (\ell !)^{\omega} \qquad \forall \ell \in \N_0,
  $$
  Then $\bs u$ is in the same type of Gevrey-class, i.e.,
  there exist $C_2(T)>0$, $\rho_2 >0$ such that for all $t \in (0,T)$
  \begin{align}
    \label{eq:apriori_wave:gevrey}
    \| \bs u^{(\ell)}(t) \|_{\VV}
    \leq C_2(T) t \rho_2^{\ell} (\ell !)^{\omega} \qquad \forall \ell \in \N_0.
  \end{align}
\end{proposition}
\begin{proof}
  \eqref{eq:apriori_wave:gevrey} is an immediate corollary by inserting the assumptions on the
  data into~\eqref{eq:apriori_wave}. 
\end{proof}

\subsection{Equivalence principle and the discrete representation formula}
Since we are interested in the apriori error of the DG-CQ approximation,
we need to relate it to the DG-timestepping result. We have the following
discrete version of \ref{prop:equiv_PDE_IE}:

\begin{lemma}
  \label{lemma:uhat_solves_helmholtz} 
  Let $\bs U^n$ be the vector of DG-timestepping
  approximations from~\eqref{eq:dg_as_timestepping}.
  Assume that there exists $C>0$ such that $\|\bs \Xi^n\|_{\MM}\leq C^n$ for all $n \in \N$. Then, the
  $Z$-Transform $\widehat{\bs U}(z):=\sum_{n=0}^{\infty}{\bs U^n z^{n}}$ solves
  for $z_0>0$ sufficiently small:
  \begin{align}
    \label{eq:uhat_solves_helmholtz}
    - \AA_{\star} \widehat{\bs U}(z) + \frac{\bs \delta(z)}{k} \widehat{ \bs U}(z) 
    &= 0,
      \qquad \BB \widehat{\bs U}(z)=\widehat{\bs \Xi}(z) \qquad \qquad \forall |z|<z_0      
  \end{align}
  with $\widehat{\bs \Xi}(z):=\sum_{n=0}^{\infty}{\bs \Xi^n z^n}$, and the matrix
  $\bs \delta(z) \in \C^{p+1\times p+1}$ is given by~\eqref{eq:def_delta}.
\end{lemma}
\begin{proof}
  The proof is verbatim to~\cite[Lemma 1.7]{MR17}.  
    See also~\eqref{eq:expl_repr_z_transform}
    where the operator $\AA_{\star}$ is replaced by a scalar, but the calculation is the same.
    The assumption on $\|\bs \Xi^{n}\|_{\MM}$ ensures the convergence of the
      $Z$-transforms of $\widehat{\bs \Xi}$ and $\widehat{\bs U}$ for sufficiently
      small $|z_0|<1/C$.
  \end{proof}

  We next note a discrete version of the representation formula, showing
  that we can construct solutions to
  ~\eqref{eq:spacetime_dg_formulation}
  using the single- and double- layer potentials and convolution quadrature.
  \begin{lemma}[Discrete representation formula]
  \label{lemma:discr_repr_formula}
  Let $ \lambda^h \in \SS^{p,0}(\TT_h) \otimes H^{-1/2}(\Gamma)$ and
  $\psi^h \in \SS^{p,0}(\TT_h) \otimes  H^{1/2}(\Gamma)$.
  Define
  \begin{align}
    \label{eq:discr_repr_formula}
    v^h&:= S(\partial_t^h )   \lambda^h -  D(\partial_t^h )  \psi^h,
         \qquad \text{and} \qquad
         \bs w^{h}:= [\partial_t^h]^{-1} \nabla v^h.
  \end{align}
  Then $\bs u^h:=(v^h,\bs w^h)$ solves~\eqref{eq:spacetime_dg_formulation1}
  with the following boundary conditions:
  \begin{align}
    \label{eq:discr_repr_formula_bc}
    \gamma^+  v^h
    &= V(\partial_t^h) \lambda^h - \big(\frac{1}{2} + K(\partial_t^h)\big)\psi^h,
      \quad \quad
      \partial_t^{h} [\gamma_\nu^+ \bs w^h]
    =
     \Big(
      \big(-\frac{1}{2} + K^t(\partial_t^h)\big)\lambda^h
      + W(\partial_t^h)\psi^h\Big).
  \end{align}
  
  In addition,
    \begin{align*}
      \lambda^h= \partial_t^{k}\llbracket{\gamma_\nu\bs w^h} \rrbracket,
      \qquad \text{and}\qquad
       \psi^h= \llbracket{\gamma v^h} \rrbracket.
    \end{align*}    

  \end{lemma}
  \begin{proof}
    Identifying the functions $v^h$, $\bs w^h$ with
    the sequence of coefficient vectors in time,
    taking the $Z$-transform of~\eqref{eq:discr_repr_formula} gives
    \begin{align*}      
      \widehat{v}^h= S\big(\frac{\bs \delta(z)}{k}\big)\hat{\bs \lambda}^h
       -  D\big(\frac{\bs \delta(z)}{k}\big)\hat{\bs \psi}^h,
      \qquad \text{and} \qquad
      \frac{\bs \delta(z)}{k}\widehat{\bs w}^h= \nabla \widehat{ v}^h.
    \end{align*}
Since the  potentials solve the Helmholtz equation, we get:
\begin{align*}
  -\laplace \widehat{v}^h + \Big(\frac{\bs \delta(z)}{k}\Big)^2 \widehat{ v}^h
  &=0,
    \qquad \text{or} \qquad
    -\nabla \cdot { \widehat{\bs w}^h} + \frac{\bs \delta(z)}{k} \widehat{v}^h=0.
\end{align*}
Therefore, $\hat{\bs U}:=(\hat{ v },\hat{\bs w})$ solves
the PDE~\eqref{eq:uhat_solves_helmholtz}.
Taking traces,
we recall the following jump conditions from Section~\ref{sect:bem_stuff}
  \begin{alignat*}{6}
    \gamma^{\pm} S\big(\nicefrac{\bs \delta(z)}{k}\big)
    &=
      V\big(\nicefrac{\bs \delta(z)}{k}\big), \qquad \qquad
    &\partial_{\nu}^{\pm} S\big(\nicefrac{\bs \delta(z)}{k}\big)
      &=\mp \frac{1}{2} + K^t\big(\nicefrac{\bs \delta(z)}{k}\big),\\
    \gamma^\pm D\big(\nicefrac{\bs \delta(z)}{k}\big)
    &= \pm \frac{1}{2} + K\big(\nicefrac{\bs \delta(z)}{k}\big), \qquad\qquad
      & \partial_{n}^{\pm} D\big(\nicefrac{\bs \delta(z)}{k}\big) &= -W\big(\nicefrac{\bs \delta(z)}{k}\big).
  \end{alignat*}
  The boundary conditions then follow by inserting the definition
  of the CQ-approximations in terms of their $Z$-transform.
\end{proof}

We are now able to show the main result of this section. Namely,
solving the discretized integral equations
and doing a post-processing
is equivalent to solving the abstract wave propagation problem~\eqref{eq:wave_eq_abstrac},
and the original densities can be recovered by applying jump operators.
This will then immediately yield convergence result by applying the abstract theory
from Section~\ref{section:abstract_wave_dg}.
\begin{theorem}[{Equivalence principle}]
  \label{thm:equivalence_principle}
  Assume that $\lambda^h$ in $\SS^{p,0}(\TT_h)\otimes H^{-1/2}(\Gamma)$ solves the
  CQ discretized version of~\eqref{eq:dirichlet}, i.e.,
  it satisfies the semidiscrete integral equations
  \begin{align}
    \label{eq:discrete_dirichlet_ie}
    \partial_t^h V(\partial_t^h) \lambda^h= \big(\frac{1}{2} - K(\partial_t^h) \big) \II\gamma^+ \duinc.
  \end{align}
  Define
  $$
  v^h:=-\partial_t^hS(\partial_t^h) \lambda^h - D(\partial_t^h) \II\duinc
  \qquad \text{and }\qquad
  \bs w^h:=[\partial_t^h]^{-1} \nabla v^h.  
  $$
  Then $\bs u^{h}:=(v^h,\bs w^h)$  is the DG-discretization of the
  abstract wave propagation problem as defined in~\eqref{eq:spacetime_dg_formulation}
  and
  $
  \lambda^h=\llbracket{\gamma_\nu\bs w^h} \rrbracket.
  $
\end{theorem}
\begin{proof}
  We only show that $\bs u^{h}:=(v^h,\bs w^h)$ solves~\eqref{eq:spacetime_dg_formulation},
  since the solution is unique by Lemma~\ref{lemma:well_posed}.

  From the discrete representation formula, we know that  $\bs u^{h}$ solves
  \eqref{eq:spacetime_dg_formulation1}.
  From the definition of $B$ in~\eqref{eq:def_operators},
  we get:
  \begin{align*}
    B \bs u^h&=(\gamma^+ v^h, \gamma^- v^h)^t
     = (\gamma^+ v^h, \llbracket{\gamma v^h}\rrbracket_{\Gamma} + \gamma^+ v^h)^t.
  \end{align*}
  Inserting the boundary conditions in~\eqref{eq:discr_repr_formula_bc},
  we get that
  $$
  \gamma^+ v^h=-\partial_t^h V(\partial_t^h) \lambda^h - \big(\frac{1}{2} + K(\partial_t^h) \big) \II\gamma^+ \duinc
      = -\II \gamma^+ \duinc
      $$
      and $\gamma^- v^h= \II \gamma^+\duinc -  \II \gamma^+ \duinc =0$.
\end{proof}

\begin{theorem}
  \label{thm:conv_traces_smooth_dirichlet}
  Assume that the boundary trace  $\gamma^+\duinc$ is in the Gevrey class $\mathcal{G}_{\omega}$, i.e., there
  exist constants $\omega \geq 1$, $C>0$ and $\rho >0$ such that
  $$
  \|\frac{d^{\ell}}{dt^{\ell}}  \uinc(t) \|_{H^{1/2}(\Gamma)}
  \leq C \rho^{\ell} (\ell !)^{\omega} \qquad \forall \ell \in \N_0, \forall t\in (0,T).
  $$
  Let $\lambda(t):=\partial_{\nu}^+ u$ denote the normal derivative to the solution~\eqref{eq:wave_eqn_classic}
  with $\BCop=\gamma^+$.
  Let $\lambda^h$ be the discrete solution~\eqref{eq:discrete_dirichlet_ie}.
  Then
  there exists a constant $\sigma>0$, depending on $\rho$ and $\omega$ such that
  \begin{align*}
    \|\lambda(t) - \lambda^h(t)\|_{H^{-1/2}(\Gamma)} +
    &\lesssim C \tstar \Big(\frac{k^{\nicefrac{1}{\omega}}}{k^{\nicefrac{1}{\omega}}+\sigma}\Big)^{\sqrt[\omega]{p}}.
  \end{align*}
\end{theorem}
\begin{proof}
  By Theorem~\ref{thm:equivalence_principle}, we can compute the error $\lambda - \lambda^h$ as
  bounding computing the jump of $\bs w - \bs w^h$.
  The result then  follows directly from Corollary~\ref{cor:conv_for_smooth_HH}
  and the regularity result of Proposition~\ref{prop:well_posedness_dirichlet}.  
\end{proof}

\section{Practical aspects}
\label{sect:practical_aspects}
In this section
we briefly discuss some of the difficulties encountered when applying the $p$-CQ in practice.
An implementation of the method can be found under \url{https://gitlab.tuwien.ac.at/alexander.rieder/pycqlib}.

\paragraph{Functional calculus}
Structurally, the method is very similar to the Runge-Kutta version of convolution quadrature.
Thus,  the techniques and fast algorithms carry over to our setting. Most
importantly the algorithms for solving systems from~\cite{BS08,B10} can be taken
almost verbatim.
One important aspect worth pointing out is that
implementing the method from Theorem~\ref{thm:discrete_equivalences}
requires the evaluation of operators
$$
K\Big(\frac{\bs \delta(z)}{k}\Big) \qquad \text{with} \qquad \bs \delta(z) \in \C^{p+1 \times p+1}.
$$
The definition of such matrix-valued arguments using the functional calculus in~\eqref{eq:def:functional_calculus}
is well suited for theoretical work, but in practice these operators
are computed by diagonalizing $\bs \delta(z)$. This was already proposed in \cite{LO93}
and for fixed order methods, the size of the matrices is independent of $k$,$n$ and thus
it can be argued that the existence of the diagonalization, as well as the
error due to numerical approximations is of lesser concern; see~\cite[Sect. 3.2]{B10} for a
discussion of this issue for some RadauIIa methods. For the $p$-method it is less
obvious that such a diagonalization procedure can be applied in a robust way that
does not impact the convergence rate. Nevertheless, in practice the
diagonalization works well, even if not covered by theory.

  We only encountered some stability problems when using $\bs \delta(0)$ for solving the convolution
  equations in a timestepping way; see \ref{sect:numerics}. This can be overcome by
  using a perturbed value $\bs \delta(\varepsilon)$.
  
Also, the cost of solving the necessary eigenvalue problem is negligible compared to the
assembly of the boundary element operators.
Nevertheless, improving on this step by either adapting the analysis or choosing a different approach more amenable to theory
is an important future step in fully understanding the method.

\paragraph{Spectral properties of $\bs \delta(z)$}
If we use the diagonalization method described above, in each timestep we need to assemble
the operators $K(\lambda)$ where $\lambda \in \sigma(\bs \delta(z)/h)$ is an eigenvalue.
If $K(\cdot)$ represents a boundary integral operator related to the wave equation,
the kernel becomes more oscillatory as the imaginary part  of $\lambda$ increases.
The computational cost for such assembly depends on the imaginary part of $\lambda$
in relation to the real part. It is thus crucial to study the spectrum of $\bs \delta(z)$ in
detail.
A first result is given in Lemma~\ref{lemma:spectrum_of_delta}, namely the
spectrum lies in the right-half plane with $\Re(\lambda) \geq \frac{1-|z|^2}{2}$,
and the imaginary part only grows as $\bigO(p^2)$.
We can also further look a the spectrum numerically. We first
compare the eigenvalues needed in order to apply an operator $K(\partial_t^h)$ where
we either use high-order DG or 3-stage RadauIIa convolution quadrature.
In order to match the later Section~\ref{sect:numerics_3d}, 
we set $T=8$ and $\delta(z)$ is sampled at the points
$z=r e^{-j\frac{\ii 2 \pi}{N+1}}$ with $N=T/h$, $r=10^{-16/{2N}}$,
$j=0,\dots, N$ as
is described in~\cite{BS08}.
In figure~\ref{fig:comp_sigma1}, we observe that the spectrum of a high order DG
method is not that different what one would expect using a standard RKCQ method,
and the method roughly falls between the two $h$-versions considered.
Since for large $\Re(\lambda)$ the damping of the kernel function
dominates the oscillatory nature, we also further look at the subset
with $\Re(\lambda)<5$, where the most challenging frequencies occur.
We observe that, again the $p$-method lies between
the two RKCQ methods. We further observe that the
bound in \ref{lemma:spectrum_of_delta} might not reveal the full structure.
Namely, the ``low damping'' regime  seems to only lie in the region $\Im(\lambda)<2p$.
Higher imaginary parts of order up to $p^2$  appear only in combination with strong damping.

\begin{figure}
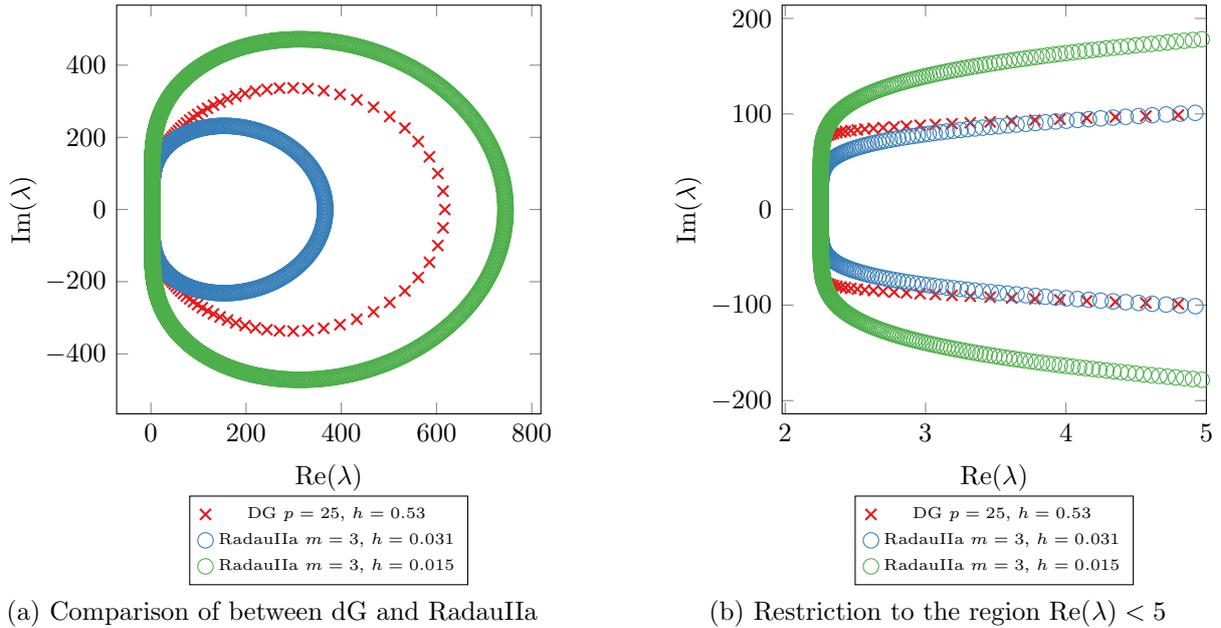

  \centering
  \begin{subfigure}{0.45\textwidth}
    \includeTikzOrEps{spectrum}
    \vspace{-1cm}
    \caption{Comparison of  between dG and RadauIIa}
    \label{fig:comp_sigma1}
  \end{subfigure}
  \hfill
  \begin{subfigure}{0.45\textwidth}
    \includeTikzOrEps{spectrum_2}
    \vspace{-1cm}    
    \caption{Restriction to the region $\Re(\lambda)<5$}
    \label{fig:comp_sigma2}
  \end{subfigure}
  \caption{Numerical investigation of $\sigma(\bs \delta(z)/h)$} 
\end{figure}

\paragraph{Assembling the matrices}
In order to create a practical method, we need to pick a basis $(\varphi_j)_{j=0}^{p}$ of the space
$\PP_p$. As suggested in \cite{SS00b}, we chose the Legendre polynomials $(L_j)_{j=0}^{p}$. They are defined by the
orthogonality properties:
\begin{align*}
  \int_{-1}^{1}{L_i(t) L_j(t) \,dt}&=\frac{2}{2j+1} \delta_{ij}.
\end{align*}
We then transplant this basis to $(0,1)$ and normalize by setting
\begin{align*}
  \varphi_j(t):=\sqrt{2j+1} \, L_{j}(2t-1).
\end{align*}
The benefit of using this basis is that the mass matrix $\mathbf{M}$ is the identity.
The assembly of the stiffness matrix $\bs S$ can also done efficiently by making use of
the recurrence relations of the Legendre polynomials, namely \cite[Eqn. (C.2.5)]{schwab_book}:
$$
(2 j +1) L_j(x) = L_{j+1}'(x) - L_{j-1}'(x).
$$
Testing with the Legendre polynomials and using the orthogonality, we can get:
\begin{align*}
  (L_{i}' , L_j)_{L^2(-1,1)} &= 0 \qquad \forall j \geq i, \\
  (L_{i}' , L_{i-1})_{L^2(-1,1)}& = 2,  \qquad \text{and}  \qquad
  (L_{i}' , L_{j})_{L^2(-1,1)} = (L_{i-2}' , L_{j})_{L^2(-1,1)}  \qquad \forall j = 0 ,\dots i-2.
\end{align*}
By the usual normalization, $L_j(1)=1$ and $L_j(-1)=(-1)^j$.
Thus  the term $\bs T(0)^t\bs T(1)^t$ can also be calculated explicitly.
Resolving the recursion, one can observe
  that $(L_{i}',L_j)_{(-1,1)}=2$ if $i>j$ and $i-j$ is odd and $(L_{i}',L_{j})_{(-1,1)}=0$
  otherwise.
  Thus, we can explicitly write down the coefficients of $\bs S$ as
  \begin{align*}
    \bs{S}_{ji}&={\sqrt{(2j+1)(2i+1)}}\Big((L_{i}',L_j)_{(-1,1)} + L_i(-1)L_j(-1)\Big) 
               = \sqrt{(2j+1)(2i+1)} \mu_{ij}
  \end{align*}
  with $\mu_{ij}:=\begin{cases}
                    1 & \text{if $i>j$}\\
                    (-1)^{i+j} & \text{else}
                  \end{cases}$.
The interpolation operator $\II$ can also be computed efficiently.
Namely, if $\II g|_{(t_n,t_{n+1})} = \sum_{j=0}^{p}{\alpha_j \varphi_j}$, then
\begin{align*}
  \alpha_j=\int_{0}^{1}{g(t_n+k\tau) \varphi_j(\tau)\,d\tau} \quad \forall j=0,\dots, p-1
  \quad \text{and} \quad
  \alpha_p=\frac{g(t_n+k) - \sum_{j=0}^{p-1}{\alpha_j\sqrt{2j+1}}}{\sqrt{2p+1}}.
\end{align*}
The first $p-1$ terms ensure the orthogonality~\eqref{eq:prop:def_II:ortho}, whereas the definition
of $\alpha_m$ ensures that $\II g$ is interpolatory at the end-point (note the normalization $L_j(1)=1$ of the Legendre polynomials). If we compute the integrals using an appropriate quadrature rule, the
cost (in terms of samples of $g$) of applying $\II$ is the similar to computing the $L^2$-projection or some interpolation procedure.

\section{Numerical examples}
\label{sect:numerics}
In this section, we show that the method presented in this article does indeed perform well in a variety
of situations.

\subsection{Scalar examples}
In order to get a good feeling for the convergence and stability of the method, we consider a scalar model problem.
Following the ideas from \cite{SV14,veit_thesis}, we choose as our domain the unit sphere, and assume that our fields
are spatially homogeneous. The constant function is an eigenfunction of all the frequency domain boundary
integral operators. Most notably
\begin{alignat*}{6}
  V(s) 1&= \frac{1-e^{-2s}}{2s}, \qquad
  &1/2+ K(s) 1&= \frac{s-1+(s+1)e^{-2s}}{2s}.
\end{alignat*}
In order to get a decent approximation to a real world problem, we solve the interior Dirichlet problem
$$
\partial_t V(\partial_t) \lambda = \big(-\nicefrac{1}{2} + K(\partial_t)\big) g.
$$
While this is not directly covered by the theory, it corresponds to  the limiting case of a
\emph{perfectly trapping sphere}. A similar model problem has for example been considered in
\cite{nonlinear_wave} as a benchmark. Compared to the exterior scattering problem, it
contains many internal reflections, modeling the trapping of waves by a more complicated
scatterer. The theory could  be easily adapted to also cover this problem, but we
decided against including it in order to preserve readability. We again refer to \cite{HQSS17}
for the functional analytic setting.

As the incident wave, we use the windowing functions                                                                                                 
$$
w_1(t):=c\begin{cases}
  e^{-\frac{1}{[t(1-t)]^\gamma}} & t \in (0,1) \\
  0 & \text{otherwise}.
\end{cases} \qquad \text{or} \qquad
w_2(t):= 
\begin{cases}t^2(1-t)^2& t \in (0,1) \\
0 & \text{otherwise}.
\end{cases}
$$
  It can be shown
  (see~\cite[Combining Ex.~1.4.9 and Prop.~1.4.6]{rodino93})
  that $w_1$ is in the Gevrey-class $G_{\omega}$ with $\omega=1+\nicefrac{1}{\gamma}$. For simplicity we used $\gamma=1$ and normalize $c$ such that $w_1(\nicefrac{1}{2})=1$. $w_2$ is analytic on $(0,1)$, but
  is non-smooth at $\pm 1$. Thus $w_2$ is not
  covered by our theory. Instead we use it as an indicator that our method also works
  for  a broader set of problems.
  The Dirichlet data is then given by
  $g(t):=w_{\bullet}(t-t_{\mathrm{lag}}) \sin(4\pi (t-t_{\mathrm{lag}}))$ with $t_{\mathrm{lag}}:=0.5$.

We applied the fast solution algorithms from \cite{BS08} as well as~\cite{B10} in
order to understand whether these can also be used in the $p$-version of convolution quadrature.
While the ``all at once'' solver \cite{BS08} worked well in all the examples considered,
the ``marching in time'' solver(``MiT'') from~\cite{B10} encountered some stability issues if
$K\big(\delta(0)/h\big)$ was evaluated directly. This can be alleviated by
also approximating this operator via the same trapezoidal rule
approximation that is used to compute the higher order convolution weights. For example,
if $\varepsilon$ denotes the precision to which $K$ can be evaluated,
a single quadrature point, i.e., using $K(\delta(\sqrt{\varepsilon})/h)$, yields good stability results up
to accuracy $\sqrt{\varepsilon}$. For higher order accuracy more quadrature points need to be used
(see also the discussion in~\cite[Sect.~5.2]{BS08} on the influence of finite precision arithmetic).
In Figure~\ref{fig:conv_scal} we used 4 quadrature points in order to make sure that
this additional approximation does not impact the convergence of the method.
The reason for this instability, and why it disappears for the modified method
  is unclear and warrants further investigation.

\begin{figure}
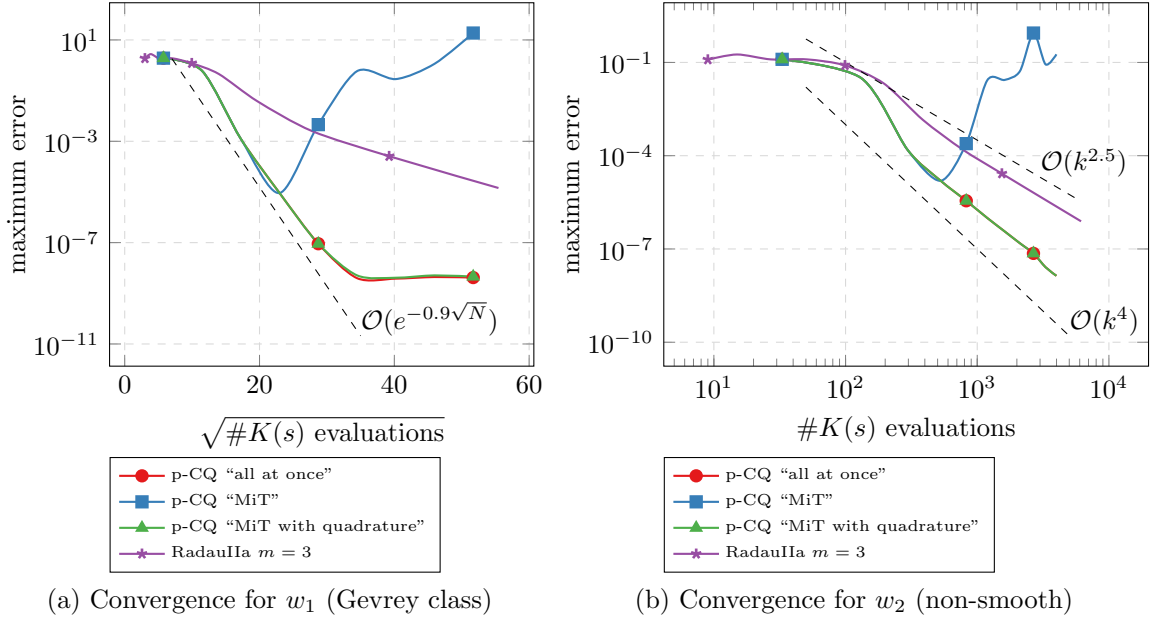

  \centering
  \begin{subfigure}{0.45\textwidth}
    \includeTikzOrEps{conv_scal}
    \vspace{-1cm}
    \caption{Convergence for $w_1$ (Gevrey class)}    
    \label{fig:conv_scal}
  \end{subfigure}
    \begin{subfigure}{0.5\textwidth}
      \includeTikzOrEps{conv_scal_w2}
      \vspace{-1cm}
    \caption{Convergence for $w_2$ (non-smooth)}
    \label{fig:conv_scal_w2}
  \end{subfigure}
  \caption{Convergence for the scalar model problem using different windowing functions}
\end{figure}
In Figure~\ref{fig:conv_scal}, we observe the expected root-exponential convergence with $\omega=2$. For the non-smooth case, we see that the convergence is algebraic instead. On the other hand, the method is nevertheless more efficient than the $3$-stage RadauIIa method considered for comparison.

\subsection{A 3d example}
\label{sect:numerics_3d}
In order to confirm that our method works for more sophisticated problems involving the discretization of boundary integral operators using a Galerkin BEM, we coupled the $p-CQ$ method
with the bempp-cl library~\cite{bempp_cl} for evaluating the boundary integral operators.
In order to perform a fast matrix-vector product instead of storing the large BEM matrices, we relied
  on the interpolated IFGF algorithm~\cite{ifgf1,ifgf2}.
  This procedure is well suited for CQ, as it is robust in terms of the frequency and
  requires minimal time to assemble the operator, thus being very fast when
  many different frequencies are required. This property also makes it
  predestined for the ``marching in time'' algorithm, which we used in order
  to get by with a very simple solving procedure. We solve
  the problem by a GMRES iterative method which is preconditioned by the
  sparse LU-factorization of the near field  of the operator. Since the
  operators $\delta(0)$ have strong damping this leads to good
  solving times.  The implementation of the IFGF algorithm and the
  modified bempp-cl version can be found under
  \url{https://github.com/arieder/ifgf} and
  \url{https://github.com/arieder/bempp-cl} respectively.
  
  In order to avoid possible stability problems due the additional matrix approximation, we
  used 2 quadrature points per timestep when approximating the CQ weights
  as described in \cite[Sect. 3.4]{book_banjai} and made sure that the
  IFGF approximation was much more accurate than the finest time
  discretization error required. In practice, one would couple all of the
  discretization parameters to achieve  balanced  error.

As the geometry we used the non-convex domain depicted in Figure~\ref{fig:model_domain} consisting of
a cutoff part of a hollow sphere. Since it is non-convex, we expect some non-trivial interactions and reflections. The mesh is fixed throughout our computations and consists of \numprint{24192} elements.

\begin{figure}[htb]
  \begin{center}
    \begin{subfigure}{0.49\textwidth}%
      \fbox{\includegraphics[height=5.5cm]{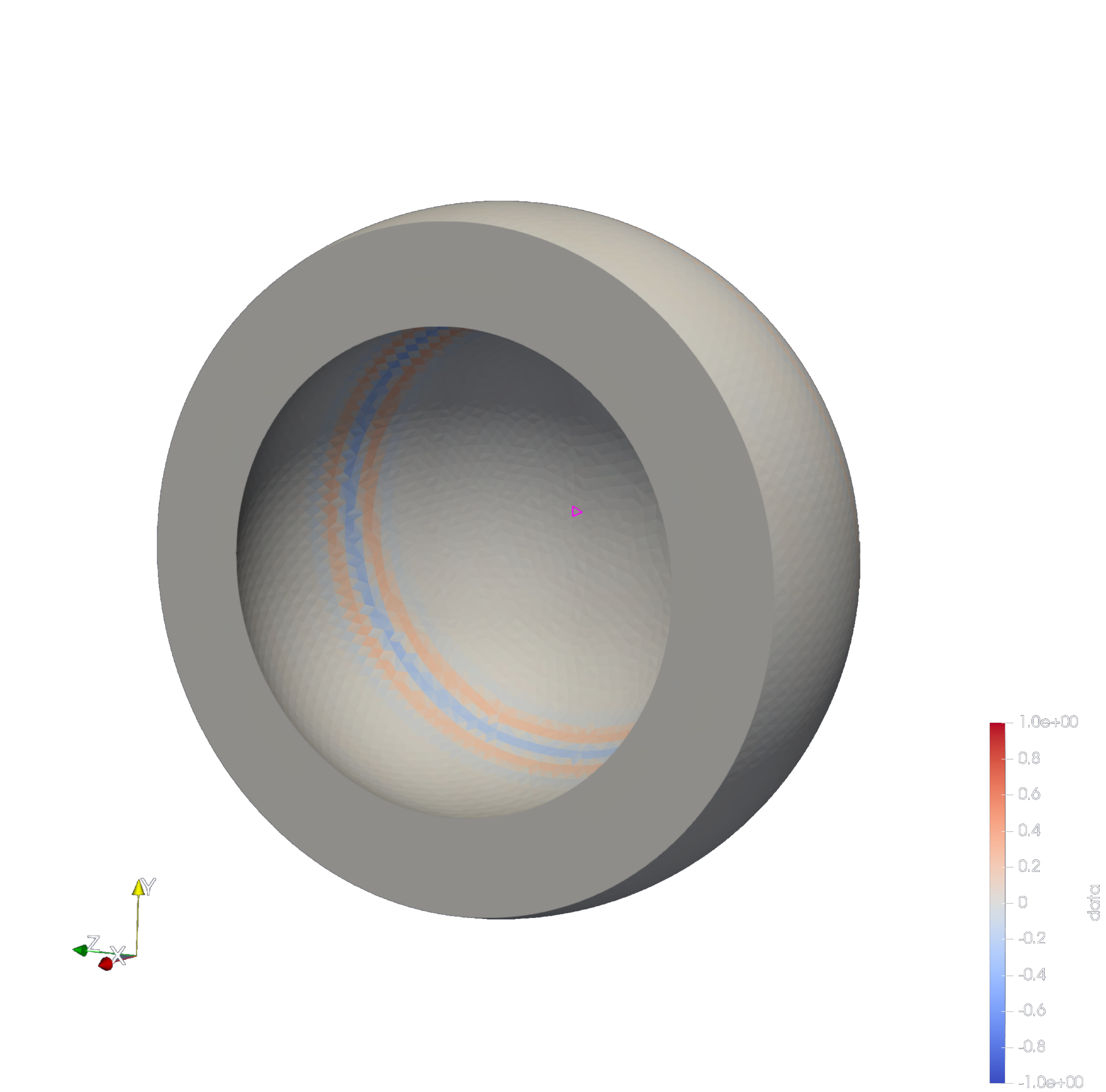}}
      \caption{Solution at $t=3$.}
      \label{fig:model_domain}
  \end{subfigure}
  \hfill
  \begin{subfigure}{0.49\textwidth}
    \includeTikzOrEps{plot_time_slice}
    \vspace{-0.75cm}
    \caption{Value of $\lambda$ at fixed cell over time.}
    \end{subfigure}
  \end{center}
  \caption{Geometry and solution for the $3d$ example}
\end{figure}

We solve the sound-soft scattering problem using an indirect ansatz. I.e., we solved the problem of finding $\lambda: [0,T] \to H^{-1/2}(\Gamma)$ such that
$$
\partial_t V(\partial_t) \lambda = \dot{g} \qquad \text{in $(0,T)$},
$$
where $\dot{g}$ is a traveling wave, i.e., $\dot{g}(t,x):=\psi(x\cdot d  - t)$ with direction $d:=(-1,0.15,0)^t$ and profile
\begin{align*}
  \psi(\xi):=w_1(\xi) \sin(8\pi \xi).
\end{align*}

We fix the number of timesteps to $15$ and compute up to the end-time $T=8$, giving a timestep size of $0.53$. We then compare the convergence of the method as we increase the polynomial order $p$.
Since an exact solution is not available for this problem, we used the finest approximation with $p=36$ as our stand-in. The $H^{-1/2}(\Gamma)$-error was approximated by $\langle V(1)[ \lambda(t)-\lambda^h(t)], \lambda(t)-\lambda^h(t) \rangle_{\Gamma}$. As $V(1)$ is elliptic this quantity is equivalent to the error.
In time, we sampled the solution at a uniform time-grid with $1024$ nodes, took the maximum over
all such times and normalized by the maximal norm of the high-precision solution.
In Figure~\ref{fig:conv_3d}, we observe the expected behavior of $\exp(-c \sqrt{p})$ predicted in Corollary~\ref{cor:conv_traces_smooth}.
In Figure~\ref{fig:conv_3d_pointwise}, we also looked at the pointwise error when the function
  $u^h=S(\partial_t^h)\lambda^h$ is sampled at discrete target points
  $x_0:=(0,0,0)^t$, $x_1:=(-1.3,0,0)^t,$ and $x_2:=(1.3,0,0)^t$. They correspond to a point
  in the middle of the hollowed sphere, one point close to the boundary and one point outside the sphere.
  We observe that the error behaves the same for all the points and similar to the error of the density $\lambda$.
  
\begin{figure}
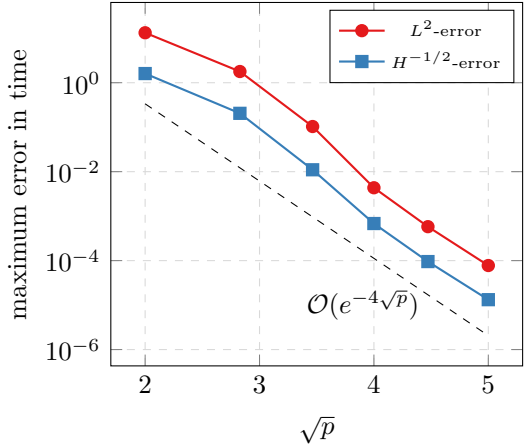
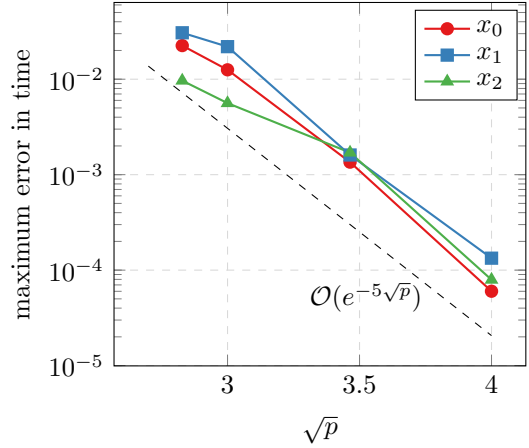

  \centering
  \begin{subfigure}{0.45\textwidth}
  \includeTikzOrEps{conv_3d}
      \vspace{-1cm}
      \caption{Convergence for density $\lambda = \partial_{\nu} u$}
    \label{fig:conv_3d}
  \end{subfigure}
    \begin{subfigure}{0.45\textwidth}
      \includeTikzOrEps{conv_3d_pw}
      \vspace{-1cm}
      \caption{Pointwise convergence for $\dot{u}=v()$ }
      \label{fig:conv_3d_pointwise}
  \end{subfigure}
  \caption{Convergence for the $3d$ scattering problem}

  \end{figure}
  In order to get a better comparison to the Runge-Kutta version of CQ, we also used approximated the same
    model problem using a 3-stage RadauIIa method and 255 timesteps, corresponding to a timestep size
    of $h=0.031$ and measured the runtime of the calculations.
    The computations were run on  our local compute server containing $2$ AMD Epyc Rome 7642 CPUs
    supporting a total 192 parallel threads.
    We collect the findings in table~\ref{tab:DG_vs_Rad}, comparing with 
    DG methods providing roughly similar accuracy. We observe that the
    DG method does indeed outperform the RadauIIa method by a significant margin.
  \begin{table}[h!]
      
    \begin{tabular}{ccc|llll}
      Method & stages & h & $L^2$-error & $H^{-1/2}$-error & Point error at $x_0$&   Runtime\\
      \hline
      RadauIIa &  3 &  0.031 & 5.32662123e-03 & 2.31877304e-03 & 1.05047250e-04  &  39h 29min 20s   \\
      DG & 12 & 0.53 & 6.15602767e-02 &  2.60370986e-02 & 1.23078338e-03 &  2h 17min 31s \\
      DG & 16 & 0.53 & 2.99804995e-03 &  1.06743307e-03 & 5.75627204e-05 & 16h 58min 39s  \\

      %
    \end{tabular}
    
    \caption{Runtime comparison of DG and Radau CQ}
    \label{tab:DG_vs_Rad}
  \end{table}

\paragraph{Acknowledgments:} 
The author gladly acknowledges
financial support  by the Austrian Science Fund (FWF) through the 
project P~36150.

  \bibliographystyle{alphaabbr}
\bibliography{literature}

\appendix
\section{Other boundary conditions and Galerkin discretization in space}
\label{sect:galerkin}
In the previous section we analyzed the direct formulation for the Dirichlet problem
in detail. Using the same techniques one can also treat all the different
discretization methods from Proposition~\ref{prop:different_model_problems}.
In addition, one can also include a Galerkin discretization in space
by adjusting the spaces and boundary operators of the abstract method
as it is done in\cite{HQSS17}. We only briefly state the results.

The Galerkin method is determined by
two closed spaces $X_{N} \subseteq H^{-1/2}(\Gamma)$
and $Y_N \subseteq H^{1/2}(\Gamma)$. We define the annihilator spaces $X_N^{\circ}$
and $Y_N^{\circ}$ as
\begin{align*}
  X_N^{\circ}&:=\big\{ \varphi \in H^{1/2}(\Gamma): \langle \mu_N,\varphi\rangle_{\Gamma}=0
  \; \;\,\forall \mu_N \in X_N\big\}, \\
  Y_N^{\circ}&:=\big\{ \psi \in H^{-1/2}(\Gamma): \langle \psi, \eta_N\rangle_{\Gamma}=0
  \; \;\,\forall \eta_N \in Y_h\big\}.
\end{align*}

The operator $\AA_{\star}$ and spaces $\HH$ and $\VV$ remain the same as for the Dirichlet problem.
The boundary space and  operators become $\MM:=X_N' \times Y_N' \times (Y_N^{\circ})' \times (X_N^\circ)'$
\begin{align}
    \label{eq:def_trace_ops_general}
  \BB: \quad \mathbb{V}   &\to \mathbb{M} :  \quad
      (v,\bs w)^t \mapsto (\gamma^+ v|_{X_N}, \tracejump{v}|_{Y_N^{\circ}}, \gamma_{\nu}^+ \bs{w}|_{Y_N},\normaltjump{\bs w}|_{X_N^{\circ}})^t,
\end{align}
where the restriction of the traces are meant in the sense of functionals,
i.e.:
\begin{align*}
  \gamma^+ v|_{X_N}
  &\in X_{N}'  \qquad \text{ is defined by }
    \quad \mu_N \mapsto \langle\mu_N,\gamma^+ v\rangle_{\Gamma}
  \quad \qquad \;\,\forall \mu_N \in X_N \subseteq H^{-1/2}(\Gamma).
\end{align*}

By adjusting the right-hand side $\bs \xi$ and the spaces $X_N$ and
$Y_N$ we can establish the discrete equivalence principle
in all of the different boundary conditions.
\begin{theorem}
  \label{thm:discrete_equivalences}
  The following equivalences hold
  for the different model problems and discretization schemes:
  \begin{enumerate}[(i)]
  \item
    \label{it:discrete_equivalences:1}
    \textbf{Dirichlet problem, direct method:}\\
    If $ \lambda^h \in \SS^{p,0}(\TT_h)\otimes X_N$ solves
    for all  $\mu^h \in \SS^{p,0}(\TT_h) \otimes X_N$:
    \begin{subequations}
      \begin{align}
        \label{eq:discrete_dirichlet_direct}
        \Big\langle \partial_t^h V(\partial_t^{h})  \lambda^h ,  \mu^h\Big\rangle_{\Gamma}
        &=\Big\langle \big(\frac{1}{2} -  K(\partial_t^h)\big) \II \duinc, \mu^h \Big\rangle_{\Gamma}.
      \end{align}
      Then  $\bs u^h:=(v^h,\bs w^h)$ defined as
      \begin{align}
        v^h&:=-\partial_t^h S(\partial_t^h )  \lambda^h  - D(\partial_t^h )  \II \duinc
             \quad \text{and}\quad
             \bs w^{h}:= [\partial_t^h]^{-1} \nabla v^h.
      \end{align}
    \end{subequations}
    solves~\eqref{eq:spacetime_dg_formulation} with
    $B$ given in \eqref{eq:def_trace_ops_general}
    using $Y_N:=\{0\}$ and $\bs \xi:=( -\II \gamma^- \duinc,\II\gamma^- \duinc,\times,0)$.
  \item
    \label{it:discrete_equivalences:2}
    \textbf{Dirichlet problem, indirect method:},
    If $ \lambda^h \in \SS^{p,0}(\TT_h)\otimes X_N$ solves
    \begin{subequations}
    \begin{align}
      \big\langle \partial_t^h V(\partial_t^{h})  \lambda^h ,  \mu^h\big\rangle_{\Gamma}
      =-\big \langle  \II \duinc, \mu_N \rangle_{\Gamma}, \qquad \forall \mu^h \in \SS^{p,0}(\TT_h)\otimes X_N.
    \end{align}
    Then $\bs u^h:=(v^h,\bs w^h)$ with
    \begin{align}
      v^h&:= \partial_t^h S(\partial_t^h )   \lambda^h \qquad \text{and}\qquad
      \bs w^{k}:= [\partial_t^h]^{-1} \nabla  v^h
    \end{align}
  \end{subequations}
  solves~\eqref{eq:spacetime_dg_formulation} with
  $B$ defined as in \eqref{eq:def_trace_ops_general}
  using    
  $Y_N:=\{0\}$ and $\bs \xi:=(-\gamma^- \duinc,0,\times,0)$.
  \item
    \label{it:discrete_equivalences:3}
    \textbf{Neumann problem, direct method},
    If $\psi^h \in \SS^{p,0}(\TT_h)\otimes Y_N$ solves    
    \begin{subequations}
    \begin{align}
      \big\langle W(\partial_t^h) \psi^h, \eta_N\big\rangle_{\Gamma}
      &=\big\langle \Big(\frac{1}{2} - K^t(\partial_t^h) \Big) \II\partial_{\nu} \uinc \big \rangle_\Gamma
        \qquad \forall \eta^h \in \SS^{p,0}(\TT_h)\otimes Y_N.
    \end{align}

    Then $\bs u^h:=(v^h,\bs w^h)$,
    \begin{align}
      v^h&:=\partial_t^h S(\partial_t^h ) \II \partial_{\nu} \uinc + \partial_t^h D(\partial_t^h) \psi^h, \qquad
      \bs w^{k}:= [\partial_t^h]^{-1} \nabla v^h
    \end{align}
  \end{subequations}
  solves~\eqref{eq:spacetime_dg_formulation} with
    $B$ defined as in \eqref{eq:def_trace_ops_general}
    using    
    $X_N:=\{0\}$ and $\bs \xi:=(\times,0,-\partial_{\nu}^- \uinc,\partial_{\nu}^- \uinc)$.
  \item
    \label{it:discrete_equivalences:4}
    \textbf{Neumann problem, indirect method}    
    If $\psi^h \in \SS^{p,0}(\TT_h)\otimes Y_h$ solves
    \begin{subequations}
    \begin{align}
    \big \langle W(\partial_t^h) \psi^h, \eta^h \rangle_{\Gamma} =
    \big \langle \II \partial_{\nu} \uinc, \eta^h \big \rangle_{\Gamma}
    \qquad \forall \eta^h \in \SS^{p,0}(\TT_h)\otimes Y_N.
    \end{align}
    Then $\bs u^h:=(v^h,\bs w^h)$,
    \begin{align}
      v^h&:= \partial_t^h D(\partial_t^h)  \psi^h, \qquad 
      \bs w^{k}:= [\partial_t^h]^{-1} \nabla v^h
    \end{align}
  \end{subequations}
  solves~\eqref{eq:spacetime_dg_formulation} with
  $B$ defined as in \eqref{eq:def_trace_ops_general}
  using    
  $X_N:=\{0\}$ and $\bs \xi:=(\times,0,-\partial_{\nu}^- \uinc,0)$.
  \end{enumerate}
\end{theorem}
\begin{proof}
  Follows completely analogously to Theorem~\ref{thm:equivalence_principle}
  by using Lemma~\ref{lemma:discr_repr_formula} and taking traces.
  For more details on how the Galerkin discretization is treated, we refer to~\cite{HQSS17}.
\end{proof}

\begin{remark}   
  The general setting can also treat other cases like symmetric Galerkin solvers,
  \cite[Sect. 5.3]{HQSS17} or more general geometric settings as in \cite{comp_scatter},
  but including those would inflate the size of this article at the detriment
  of readability.
\end{remark}

As we did before in Theorem~\ref{thm:conv_traces_smooth_dirichlet}, we can now
use the abstract theory to show exponential convergence also for the other
discretization schemes.
\begin{corollary}
  \label{cor:conv_traces_smooth}
  Assume that the boundary trace  $\mathbf{\BCop} \uinc$ is in the Gevrey class $\mathcal{G}_{\omega}$, i.e., there
  exist constants $\omega \geq 1$, $C>0$ and $\rho >0$ such that
  $$
  \|\frac{d^{\ell}}{dt^{\ell}} \BCop \uinc(t) \|_{H^{\pm 1/2}(\Gamma)}
  \leq C \rho^{\ell} (\ell !)^{\omega} \qquad \forall \ell \in \N_0, \forall t\in (0,T),
  $$
  (the norm depends on whether $\BCop$ is the Dirichlet or Neumann trace).

  Let $\lambda(t):=\partial_{\nu}^+ u$ and $\psi(t):=\gamma^+ u$.
  Let $\lambda^h$ and $\psi^h$ be computed in
  one of the ways of Theorem~\ref{thm:discrete_equivalences}. Then
  there exists a constant $\sigma>0$, depending on $\rho$ and $\omega$ such that
  \begin{align*}
    \|\lambda(t) - \lambda^h(t)\|_{H^{-1/2}(\Gamma)} +
    \frac{1}{\tstar}\|\psi(t) - \psi^h(t)\|_{H^{1/2}(\Gamma)}
    &\lesssim C_{\bs u} \tstar \Big(\frac{k^{\nicefrac{1}{\omega}}}{k^{\nicefrac{1}{\omega}}+\sigma}\Big)^{\sqrt[\omega]{p}}.
  \end{align*}
\end{corollary}
\begin{proof}
  By Theorem~\ref{thm:discrete_equivalences}, we can compute the error $\lambda - \lambda^h$ as
  well as $\psi - \psi^h$ by computing the jump of $\bs w - \bs w^h$
  and  $u - [\partial_t^h]^{-1}v^h$ respectively.
  The statement for $\lambda$ then follows directly from Corollary~\ref{cor:conv_for_smooth_HH}
  and the continuity of  the trace operator in $H(\operatorname{div})$.
  For $\psi$,  
  we bound using the trace theorem and $[\partial_t^h]^{-1} \nabla v^h=\bs w^h$:
  $$\big\|\tracejumpsmall[\big]{\big(u - [\partial_t^h]^{-1}v^h\big)}\big\|_{H^{1/2}(\Gamma)}
  \lesssim \|u - [\partial_t^h]^{-1}v^h\|_{L^2(\R^d)} + \|\nabla u - \bs w^h\|_{L^2(\R^d \setminus \Gamma)}.
  $$
  Thus, what is left to bound is the $L^2$-term. Since we are only interested
  in exponential convergence we can proceed rather crudely. By Lemma~\ref{lemma:integration},
  and using the stability of the operator $\II$ from Lemma~\ref{prop:def_II}(\ref{it:prop:def_II:2}):
  \begin{align*}
    \|u - [\partial_t^h]^{-1}v^h\|_{L^\infty{((0,T);L^2(\R^d))}}
    &=\|u - [\II \partial_t^{-1} v^h]\|_{L^\infty{((0,T);L^2(\R^d))}} \\
    &\leq \|u - \II u\|_{L^\infty{((0,T);L^2(\R^d))}} + \| \II ( u - \partial_t^{-1} v^h)\|_{L^\infty{((0,T);L^2(\R^d))}} \\
    &\lesssim
      \|u - \II u\|_{L^\infty{((0,T);L^2(\R^d))}} \\
    & \;\, + \|\partial_t^{-1} (v -  v^h)\|_{L^\infty{((0,T);L^2(\R^d))}} 
     + k\|v -  v^h\|_{L^{\infty}((0,T), L^2(\R^d))} \\
    &\lesssim
      \|u - \II u\|_{L^\infty{((0,T);L^2(\R^d))}}
    + (T+k)\|v -  v^h\|_{L^\infty{((0,T);L^2(\R^d))}}.
  \end{align*}
  By the already established results, all terms are exponentially small, which concludes the proof.
\end{proof}

\end{document}